\def\draft{n}
\documentclass[12pt]{amsart}
\usepackage[headings]{fullpage}
\usepackage{amssymb,epic,eepic,epsfig}
\usepackage{epstopdf}
\usepackage{graphicx}
\usepackage{texdraw}
\usepackage{url}
\usepackage[bookmarks=true,%
    colorlinks=true,%
    linkcolor=blue,%
    citecolor=blue,%
    filecolor=blue,%
    menucolor=blue,%
    urlcolor=blue,%
    breaklinks=true]{hyperref}

\newtheorem{theorem}{Theorem}[section]
\theoremstyle{definition}
\newtheorem{proposition}[theorem]{Proposition}
\newtheorem{lemma}[theorem]{Lemma}

\newtheorem{remark}[theorem]{Remark}

\newtheorem{conjecture}[theorem]{Conjecture}
\newtheorem{question}[theorem]{Question}

\def\printname#1{
        \if\draft y
                \smash{\makebox[0pt]{\hspace{-0.5in}
                        \raisebox{8pt}{\tt\tiny #1}}}
        \fi
}

\newcommand{\psdraw}[2]
         {\begin{array}{c} \hspace{-1.3mm}
        \raisebox{-4pt}{\epsfig{figure=draws/#1.eps,width=#2}}
        \hspace{-1.9mm}\end{array}}

\newlength{\standardunitlength}
\catcode`\@=11
\long\def\@makecaption#1#2{%
     \vskip 10pt

\setbox\@tempboxa\hbox{
       \small\sf{\bfcaptionfont #1. }\ignorespaces #2}%
     \ifdim \wd\@tempboxa >\captionwidth {%
         \rightskip=\@captionmargin\leftskip=\@captionmargin
         \unhbox\@tempboxa\par}%
       \else
         \hbox to\hsize{\hfil\box\@tempboxa\hfil}%
     \fi}
\font\bfcaptionfont=cmssbx10 scaled \magstephalf
\newdimen\@captionmargin\@captionmargin=2\parindent
\newdimen\captionwidth\captionwidth=\hsize
\catcode`\@=12

\def\lbl#1{\label{#1}\printname{#1}}

\def\BN{\mathbb N}
\def\BZ{\mathbb Z}

\def\BQ{\mathbb Q}

\def\calP{\mathcal P}
\def\fl{\mathrm{fl}}

\def\longto{\longrightarrow}

\def\calG{\mathcal{G}}

\def\deg{\mathrm{deg}}

\def\coeff{\mathrm{coeff}}

\def\lb{[\![}
\def\rb{]\!]}
\def\pl{\mathrm{pl}}

\begin{document}



\title[Flag algebras and the stable coefficients of the Jones polynomial]{
Flag algebras and the stable coefficients of the Jones polynomial}

\author{Stavros Garoufalidis}
\address{School of Mathematics \\
         Georgia Institute of Technology \\
         Atlanta, GA 30332-0160, USA \newline
         {\tt \url{http://www.math.gatech.edu/~stavros }}}
\email{stavros@math.gatech.edu}
\author{Sergey Norin}
\address{Department of Mathematics and Statistics \\
         Burnside Hall, 805 Sherbrooke West \\
         Montreal, QC, H3A 2K6 CANADA \newline
         {\tt \url{http://www.math.mcgill.ca/snorin }}}
\email{snorin@math.mcgill.ca}
\author{Thao Vuong}
\address{School of Mathematics \\
         Georgia Institute of Technology \\
         Atlanta, GA 30332-0160, USA \newline
         {\tt \url{http://www.math.gatech.edu/~tvuong }}}
\email{tvuong@math.gatech.edu}
\thanks{
1991 {\em Mathematics Classification.} Primary 57N10. Secondary 57M25.
}

\date{April 22, 2015}


\begin{abstract}
We study the structure of the stable coefficients of the
Jones polynomial of an alternating link. We start by identifying the first four
stable coefficients with polynomial invariants of a (reduced) Tait graph of
the link projection. This leads us to introduce a free polynomial algebra
of invariants of graphs whose elements give invariants of alternating links
which strictly refine the first four stable coefficients. We conjecture that
all stable coefficients are elements of this algebra, and give experimental
evidence for the fifth and sixth stable coefficient. We illustrate our
results in tables of all alternating links with at most 10 crossings and all
irreducible planar graphs with at most 6 vertices.
\end{abstract}

\maketitle

\tableofcontents



\section{Introduction}
\lbl{sec.intro}

\subsection{The stable coefficients of the Jones polynomial}

The paper identifies a quantum knot invariant (the third stable coefficient
of the Jones polynomial of an alternating link) with a polynomial of
induced graphs countings of a plane graph (a Tait graph of the alternating
link). Our input is a $q$-hypergeometric series $\Phi_G(q) \in \BZ[\![q]\!]$ 
that is associated to a plane rooted graph. $\Phi_G(q)$  encodes the stable 
coefficients of the Jones polynomial of the corresponding alternating link. 
A combinatorial analysis of the coefficient of $q^3$ of $\Phi_G(q)$ is the 
focus of our paper; see Theorem~\ref{thm.1} below.

Perhaps more interesting than the explicit formula given in 
Equation~\eqref{eq.coeff3} is the fact that it is a polynomial
in induced graph countings of $G$. This new phenomenon, not observed
in the previously known coefficients of $q^k$ of $\Phi_G(q)$ for $k=0,1,2$.
This discovery leads on the one hand to atomic knot invariants (discussed 
after Conjecture~\ref{conj.3}), and on the other hand to the algebra of 
graph induced countings, an interesting object on its own right.

The aim of our paper is to study this unexpected discovery between the
algebra of graph induced countings and the stable coefficients of the Jones
polynomial.

Although the results of our paper concern quantum knot invariants, they 
require no prior knowledge of knot theory nor familiarity with the colored 
Jones polynomial of a knot or link. As a result, we will not recall the
definition of the {\em Jones polynomial} $J_L(q) \in \BZ[q^{\pm 1/2}]$
of a knot or link $L$ in 3-space, which may be found in several 
texts~\cite{Jo,Tu,Tu:book,Kauffman}. 
A stronger invariant is the {\em colored Jones polynomial}
$J_{L,n}(q) \in \BZ[q^{\pm 1/2}]$, where $n \in \BN$, which essentially encodes
the Jones polynomial of a link and its parallels \cite[Cor.2.15]{Kirby}. 
When $L$ is an {\em alternating link}, (i.e., a link with an alternating
planar projection~\cite{Kauffman}) the coefficients of the (shifted) colored 
Jones polynomial $\hat J_{L,n}(q) \in 1 + q\BZ[q]$ stabilize, in the 
following sense: for every $k \in \BN$, the coefficient of
$q^k$ in $\hat J_{L,n}(q) \in \BZ[q]$ is independent of $n$ for $n>k$. Those
stable coefficients assemble to a formal power series
$\Phi_L(q) \in \BZ[\![q]\!]$, where $\BZ[\![q]\!]$ denotes the ring of formal 
power series in a variable $q$ with integer coefficients.

The existence of $\Phi_L(q)$ was given in~\cite{GL1,AD} and 
a presentation as a $q$-hypergeometric series (of Nahm type) which depends
only on a plane graph (a Tait graph of $L$) was given in \cite{GL1}. 
For a rooted plane graph $G$, $\Phi_G(q)$ is given by a $q$-hypergeometric 
sum of the form
\begin{equation}
\lbl{eq.defphi}
\Phi_G(q)=(q)_\infty^{c_2}\sum\limits_{(a,b)}(-1)^{B(a,b)}
\frac{q^{\frac{1}{2}A(a,b)+\frac{1}{2}B(a,b)}}{\prod\limits_{(p,v)}(q)_{a_p+b_v}}
\end{equation}
where the sum is over the set of all admissible states $(a,b)$ of $G$, 
(i.e., admissible colorings of the faces and the vertices of $G$ by integers)
and the product is over the set of all corners $(p,v)$ of $G$.
Here, $(q)_m=(1-q)(1-q^2) \dots (1-q^m)$ for a natural number $m$ and 
$(q)_\infty=(1-q)(1-q^2)(1-q^3)\dots$.
For a detailed explanation of the notation and terminology, 
see Section \ref{sub.phiGq}. 

We will denote by $\phi_{G,k}$ (resp., $\phi_{L,k}$)
the coefficient of $q^k$ in $\Phi_G(q)$ (resp., $\Phi_L(q)$), and we will
often call it the $k$-th stable coefficient of $G$ (resp., $L$).

In \cite{DL} the first three stable coefficients $\phi_k:G \mapsto \phi_{G,k}$
for $k=0,1,2$ were expressed in terms of the number of
vertices, edges and 3-cycles of $G$. The proof used  properties
of the Kauffman bracket skein module. An independent proof was given in
\cite{GV}. To express the answer, and to motivate the polynomial algebra
$\calP$ introduced below, consider the elements
$c_1,c_2,c_3 \in \calP$ given by

\begin{equation}
\lbl{eq.a3}
(c_1, c_2, c_3)=( \lb \bullet \rb, \lb \psdraw{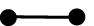}{0.3in} \rb,
\lb \psdraw{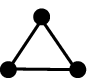}{0.3in} \rb) \,.
\end{equation}
$c_1,c_2,c_3$ count the number of vertices, edges and triangles in a
graph $G$. Then, we have \cite{DL}
\begin{equation}
\lbl{eq.coeff02}
(\phi_{0},\phi_{1},\phi_{2})=
\left(1,c_1-c_2 -1, \frac{1}{2}\left( (c_1-c_2)^2 -2 c_3 - c_1 +
c_2 \right)\right) \,.
\end{equation}

It is natural to ask for a formula for the next coefficient $\phi_{3}$.
The answer is given in Theorem \ref{thm.1} below. What's more,
Theorem \ref{thm.1}
\begin{itemize}
\item[(a)]
motivates us to
introduce the algebra $\calP$ of polynomial invariants of graphs, in the
spirit of flag algebras of \cite{Razborov}. $\calP$ turns out to be a
free polynomial algebra, see Theorem \ref{thm.P1}.
\item[(b)]
shows that $\phi_3$ is determined by $\phi_k$ for $k \leq 2$ and
$-c_{41}+2c_{42}$. The latter is an integer linear combination of the refined
alternating link invariants $c_{41}, c_{42}$; see Proposition \ref{prop.P2}
\item[(c)] motivates us to write $\Phi(q)$ as an infinite product and
conjecture that its exponents are linear forms on the set of irreducible
planar graphs, see Conjecture \ref{conj.4} and its explicit form,
Conjecture \ref{conj.3}. The latter is verified by explicit computation
for all alternating links with at most 10 crossings and all irreducible
graphs with at most 7 vertices.
\item[(d)] raises the question of how Rozansky's
categorification~\cite{Rozansky} $\Phi_L(t,q)$ of $\Phi_L(q)=\Phi_L(-1,q)$ 
can further refine Conjecture \ref{conj.4}. Since this categorification is not
yet effectively computable, we cannot make this question more precise.
\end{itemize}

\subsection{An algebra $\calP$ of polynomial invariants of graphs}
\lbl{sub.poly}

Let $\calG$ denote the set of simple finite graphs, i.e., non-embedded
graphs with no loops and no multiple edges, and unlabeled vertices
and edges. For $H$ and $G$ in $\calG$, an {\em embedding} $f: H \to G$ is an
injection $f: V(H) \hookrightarrow V(G)$ (where $V(G)$ denotes the set of
vertices of $G$) such that for every $v, v' \in V(H)$ $(v,v')$ is an edge
of $H$ if and only if $(f(v),f(v'))$ is an edge of $G$. Let $i(H,G)$
denote the number of embeddings of $H$ in $G$, divided by the number of
automorphisms of $H$. Varying $G$, we get a function
$[H]: \calG \to \BN$ given by $G \in \calG \mapsto [H](G)=i(H,G)$.
The {\em degree} of $[H]$ is the number of vertices of $H$. Let
$[\calG]$ denote the set $\{[H] \,\, | H \in \calG\}$. Likewise
we define $[\calG^c]$ where $\calG^c$ is the set of connected graphs.
$\calP$ denotes the $\BQ$-vector space on the set $[\calG]$.

\begin{proposition}
\lbl{prop.P1}
\rm{(a)} $\calP$ is a commutative algebra. In fact,
\begin{equation}
\lbl{eq.mult}
[H_1][H_2] = \sum_H c_H [H]
\end{equation}
where $H$ is a graph on at most $|V(H_1)|+|V(H_2)|$ vertices and
$c_H$ is the number of ordered pairs of induced subgraphs $(F_1,F_2)$ of $H$
(possibly sharing some vertices) such that $F_i$ is isomorphic to $H_i$ for
$i=1,2$ and moreover $V(F_1) \cup V(F_2) = V(H)$.
\newline
\rm{(b)} It follows that $\calP$ is a quotient of the polynomial algebra
on $[\calG^c]$.
\end{proposition}

Equation \eqref{eq.mult} shows that the structure constants of the
multiplication in $\calP$ are natural numbers. For instance we have:
$$
\frac{1}{2}([ \bullet ]^2- [ \bullet ]) =  [ \psdraw{myG2}{0.3in} ] 
+  [ \bullet \,\, \bullet ]
$$
This holds since both sides of the above equation evaluated on $G \in \calG$ 
equal to the number of pairs of vertices of $G$ and such a pair is either
connected by an edge or not. More generally, if $H$ is a graph on $k$ vertices
then
$$
[H][\bullet ]=k[H]+\sum c_F[F]
$$
where the sum is over all graphs $F$ on $k+1$ vertices and $c_F$
is equal to the number of induced subgraphs of $F$ isomorphic to $H$.

\begin{theorem}
\lbl{thm.P1}
$\calP$ is a free polynomial algebra on the set $[\calG^c]$.
\end{theorem}

Real valued functions on $\calG$ are also called {\em graph parameters}
and linear combinations of graphs are also called {\em quantum graphs}
in the context of graph theory. The algebra $\calP$ is reminiscent to the
{\em flag algebras} of graph theory \cite{Razborov}.

Since alternating links involve planar graphs only, let $\calG^{\pl}$ denote
the set of simple planar graphs. For $H \in \calG^{\pl}$, we denote by
$\lb H \rb $ the restriction of the function $[H]: \calG \to \BN$ to
$\calG^{\pl} \subset \calG $, and $\calP^{\pl}$ the vector space generated
by $\lb H \rb $ for $H \in \calG^{\pl}$. $\calP^{\pl}$ is also an algebra.
The structure of the algebra $\calP^{\pl}$ is an interesting and
challenging problem.

\subsection{A formula for $\phi_{3}$}
\lbl{sub.results}

Let $c_{4,i}=\lb Gv^4_i \rb $ for $i=1,2$ where $Gv^4_i$ are shown in Figure
\ref{f.Gv4}.

\begin{figure}[htpb]
$$
\psdraw{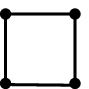}{0.5in} \qquad \psdraw{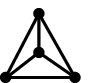}{0.5in}
$$
\caption{The irreducible planar graphs $Gv^4_1$ (left) and
$Gv^4_2$ (right) with 4 vertices.}
\lbl{f.Gv4}
\end{figure}

\begin{theorem}
\lbl{thm.1}
We have:
\begin{equation}
\lbl{eq.coeff3}
\phi_{3} =
c_{41} - 2 c_{42} + \frac{c_2}{6} + c_3 c_2 - \frac{c_2^3}{6}
- \frac{c_1}{6} - c_3 c_1 + \frac{c_2^2 c_1}{2} -
\frac{ c_2 c_1^2}{2} + \frac{c_1^3}{6} \,.
\end{equation}
Equations \eqref{eq.coeff02} and \eqref{eq.coeff3} are equivalent to
\begin{equation}
\lbl{eq.phiall}
\Phi(q) = (1-q)^{1-c_1+c_2}(1-q^2)^{c_3} (1-q^3)^{c_3 -c_{41}+2 c_{42}} + O(q^4) \,.
\end{equation}
\end{theorem}

\subsection{A conjecture for $\phi_4, \phi_5$ and $\phi_k$}
\lbl{sub.conj.phik}

A comparison of Equations \eqref{eq.coeff3} and \eqref{eq.phiall}
suggests us to write $\Phi_G(q)$ as an infinite product
\begin{equation}
\lbl{eq.phiL}
\Phi(q)=(1-q)^{1-c_1+c_2} \prod_{k=2}^\infty (1-q^k)^{C_k}
\end{equation}
where $C_k(G) \in \BZ$ for all $k$. This is possible by the following Lemma.
\begin{lemma}
\lbl{lem.abn}
For every sequence of integers $(a_n)$ there exists a sequence of integers
$(b_n)$ such that
\begin{equation}
\lbl{eq.abn}
1+\sum_{n=1}^\infty a_n q^n = \prod_{n=1}^\infty (1-q^n)^{b_n} \,.
\end{equation}
\end{lemma}

\begin{proof}
Define $b_n$ inductively by
$$
\coeff( (1+\sum_{m=1}^\infty a_m q^m) \prod_{k=1}^{n-1} (1-q^k)^{-b_k},
q^n) = -b_n \,.
$$
Given $b_n$ as above, by induction on $n$ it follows that for all $n>0$ 
we have
$$
(1+\sum_{m=1}^\infty a_m q^m) \prod_{k=1}^{n-1} (1-q^k)^{-b_k} \in 1 + q^n 
\BZ[\![q]\!] \,.
$$
Letting $n$ go to infinity, it follows that the right hand side of the
above equation is $1$, hence Equation \eqref{eq.abn} follows.
\end{proof}
Theorem \ref{thm.1} gives an expression
for $C_k$ for $k=2,3$. To phrase our conjecture for $C_k$ for $k=4,5$,
recall the notion of an irreducible planar graph from \cite{GV}.
The latter is a planar graph which is not a vertex connected sum or an
edge connected sum of planar graphs as in Figure \ref{f.connected.sum}.
The table of irreducible planar graphs with at most 10 edges is given
in Figures \ref{f.graphs345}, \ref{f.graphs67}, \ref{f.graphs8},
\ref{f.graphs9} and \ref{f.graphs10}, and with at most 6 vertices is
given in Figures \ref{f.Gv4}, \ref{f.Gv5} and \ref{f.Gv6}.

\begin{figure}[htpb]
$$
\psdraw{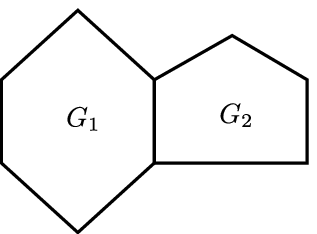}{1.3in} \qquad\qquad
\psdraw{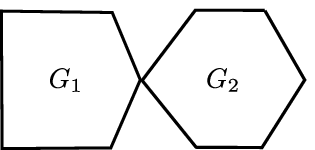}{1.3in}
$$
\caption{A vertex connected sum (on the left) and an edge-connected sum
on the right.}
\lbl{f.connected.sum}
\end{figure}

\begin{conjecture}
\lbl{conj.3}
We conjecture that
\begin{align}
\lbl{eq.C4}
C_4 &= c_3 - c_{41} + 5 c_{42} + c_{51} - c_{52} - 2 c_{53} - 3 c_{54} \\
\lbl{eq.C5}
C_5 &=c_3-c_{41}+12 c_{42}+c_{51}-4 c_{53}-9 c_{54}\\
\notag
& -c_{61}+c_{62}  -2c_{63}-c_{64}+2 c_{65}+3 c_{66}+4 c_{68}-4 c_{69}\\
\notag
& +2c_{610}+c_{611}-3 c_{612}+4 c_{613}+c_{614}-5 c_{616}-16 c_{618}+c_{619}
\end{align}
where  $c_{j,i}=\lb Gv^j_i \rb $ and $Gv^5_i$ and $Gv^6_i$ are
{\em irreducible} planar graphs with $5$ and $6$ vertices shown in Figures
\ref{f.Gv5} and \ref{f.Gv6}.
\end{conjecture}
Independently of the above conjecture, each term that appears in the right
hand side of Equations \eqref{eq.C4}-\eqref{eq.C5} is an alternating
link invariant; see Proposition \ref{prop.P2} below.

The expression for $C_4$ and $C_5$ is the unique linear combination
of irreducible planar graphs with 5 and 6 vertices
(and this is how it was found) which fits the stable
coefficients of the Jones polynomial of all alternating links with at
most 10 crossings and all alternating links whose reduced Tait graph has
at most 6 vertices. For details, see Section \ref{sec.compute}.

The reader may observe that the graph $Gv^5_5$ is missing from $C_4$. This
motivates the following question.

\begin{question}
\lbl{que.1}
Is it true that
$$
(\Phi_{G^6_1})^2 = \Phi_{G^9_1} \Phi_{G^3_0}
$$
A direct computation confirms this up to $O(q^{31})$.
\end{question}

\begin{conjecture}
\lbl{conj.4}
For all $k \geq 2$, $C_k$ are linear forms with integer
coefficients on the set of irreducible planar graphs with at most $k+1$
vertices.
\end{conjecture}

The above conjecture has an equivalent formulation.

\begin{conjecture}
\lbl{rem.conj4}
$\Phi$ is multiplicative under vertex and edge connected sum, and for
every connected irreducible planar graph $H$ there exist
$\Psi_H(q) \in 1 + q^{\deg(H)-1}\BZ[\![q]\!]$
such that
\begin{equation}
\lbl{eq.phiL2}
\Phi(q)=(1-q)^{1-c_1+c_2} \prod_{H} \Psi_H(q)^{ \lb H \rb}
\end{equation}
where the product is taken over the set of irreducible planar graphs.
\end{conjecture}


\section{The algebra $\calP$}
\lbl{sec.calP}

\subsection{Proof of Proposition \ref{prop.P1}}
\lbl{sub.proof.P1}

A \emph{subgraph of a graph $G$ induced by $S \subseteq V(G)$} is a graph
$G[S]$ such that $V(G[S])=S$ and two vertices in $S$ are joined by an edge
in $G[S]$ if and only if they are joined by an edge in $G$. The value
$i(H,S)$ can be equivalently defined as the number of sets
$S \subseteq V(G)$ such that $G[S]$ is isomorphic to $H$.

To show that (\ref{eq.mult2}) holds we need to show that
\begin{equation}
\lbl{eq.mult2}
i(H_1,G)i(H_2,G)= \sum_{H}c_H i(H,G)
\end{equation}
for every graph $G$. Note that $i(H_1,G)i(H_2,G)$ equals the number of
pairs $(S_1,S_2)$ of subsets of $V(G)$ such that
$G[S_i]$ is isomorphic to $H_i$ for $i=1,2$. We claim that  for a fixed graph
$H$ the number of
pairs as above, such that $G[S_1 \cup S_2]$ is isomorphic to $H$, is equal to
$c_H i(H,G)$.  The equation (\ref{eq.mult2}) immediately follows from this
claim. The claim holds as
the number of sets $S \subseteq V(G)$ such that $G[S]$ is isomorphic to $H$
is equal to $i(H,G)$. Further, for given $S \subseteq V(G)$ the number of
pairs $(S_1,S_2)$ defined above with $S=S_1 \cup S_2$ equals $c_H$, by
definition.
\qed

\subsection{Proof of Theorem \ref{thm.P1}}
\lbl{sub.proof.thmP1}

The proof of the theorem is derived from the results of \cite{ELS}. We start
by introducing the additional notation, which will allow us to state the
necessary results.
Let $$\gamma(H,G)=i(H,G)/\binom{|V(G)|}{|V(H)|}.$$
Let $k$ be a fixed integer and let $H_1,H_2,\ldots,H_m$ be all connected
graphs with $|V(H_i)| \leq k$. Given a graph $G$ define a vector
$$
\mathbf{\gamma}(k,G)=(\gamma(H_1,G), \gamma(H_2,G), \ldots, \gamma(H_m,G))\,.
$$
Let $S_k$ be defined as the set of all vectors $\mathbf{v} \in \mathbb{R}^m$
such that there exists an infinite sequence of graphs
$G_1,G_2,\ldots G_n,\ldots,$ such that
$|V(G_n)| \to \infty$ and $\mathbf{\gamma}(k,G) \to \mathbf{v}$. The
following lemma follows immediately from  \cite[Theorems 1 and 3]{ELS}.

\begin{lemma}
\lbl{lem:ELS}
Let $k$ be a positive integer, let $m$ be the number of connected graphs on 
at most $k$ vertices and let $S_k \subseteq  \mathbb{R}^m$ be as defined above.
Then $S_k$ contains an $m$-dimensional ball of positive radius.
\end{lemma}

We are now ready to prove Theorem~\ref{thm.P1}.

\begin{proof}(of Theorem~\ref{thm.P1})
Let $k$ be a positive integer and let $H_1,H_2,\ldots,H_m$ be all connected
graphs on at most $k$ vertices, as before. It suffices to show that
for every $p \in \mathbb{R}[x_1,x_2,\ldots,x_m]$, $p \not \equiv 0$ (i.e.,
$p$ not identically zero), we have
$p([H_1],[H_2],\ldots,[H_m]) \not \equiv 0$.
Suppose for a contradiction that for some polynomial
$p_0 \in \mathbb{R}[x_1,x_2,\ldots,x_m]$, $p_0 \not \equiv 0$ we have
$$
p_0(i(H_1,G),i(H_2,G), \ldots,i(H_m,G)) =0
$$ for every graph $G$.
As $i(H,G)=\binom{|V(G)|}{|V(H)|}\gamma_H(G),$ there exists an
$(m+1)$-variable polynomial $p_1 \in \mathbb{R}[x_1,x_2,\ldots,x_m,y]$,
$p_1 \not \equiv 0$ such that for each graph $G$ we have
\begin{align*}
p_0&(i(H_1,G),i(H_2,G), \ldots,i(H_m,G)))=
p_1(\gamma(H_1,G),\gamma(H_2,G), \ldots,\gamma(H_m,G),|V(G)|) \\
&(\;=p_1(\mathbf{\gamma}(k,G), |V(G)|) \;).
\end{align*}
Let $$p_1(x_1,x_2,\ldots,x_m,y)=\sum_{i=1}^tr_i(x_1,\ldots,x_m)y^{i},$$
Suppose without loss of generality that $r_t$ is not identically zero. We
claim that $r_t$ is identically zero on $S_k$, in contradiction with
Lemma~\ref{lem:ELS}.

To prove the claim, consider $\mathbf{v} \in S_k$ and let
$G_1,G_2,\ldots G_n,\ldots$ be a sequence of graphs such that
$|V(G_n)| \to \infty$ and $\mathbf{\gamma}(k,G_n) \to \mathbf{v}$, as in the
definition of $S_k$. Let 
$f(G_n)=p_1(\mathbf{\gamma}(k,G_n),|V(G_n)|)/|V(G_n)|^t$.
Clearly, $\lim_{n \to \infty}f(G_n) = r_t(\mathbf{v})$. On the other hand,
$f(G_n)=0$ for every $n$ by the choice of $p_1$. It follows that
$r_t(\mathbf{v})=0$, as desired. This establishes the claim and the theorem.
\end{proof}

\subsection{A subalgebra $\calP^{\fl}$ of $\calP$}
\lbl{sub.flype}

In this section we introduce a subalgebra $\calP^{\fl}$ of $\calP$ which
is motivated by knot theory. Consider a {\em flype move} on a graph
shown in Figure \ref{f.flyping.graphs}.

\begin{figure}[htpb]
$$
\psdraw{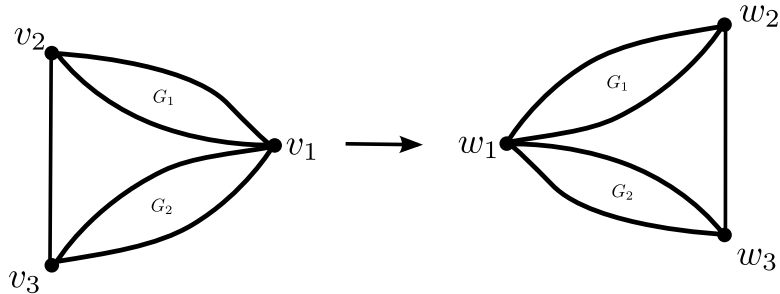}{4in}
$$
\caption{A flype move on a planar graph.}
\lbl{f.flyping.graphs}
\end{figure}
The importance of the flype move is Tait's Conjecture proven by
Menasco-Thisthlethwaite \cite{MT}: every two reduced $S^2$
projections of an alternating link are connected by a sequence of flype moves.
Closely related to a flype move is a {\em Whitney flip} move \cite{Whitney},
illustrated in Figure \ref{f.flip}.

\begin{figure}[htpb]
$$
\psdraw{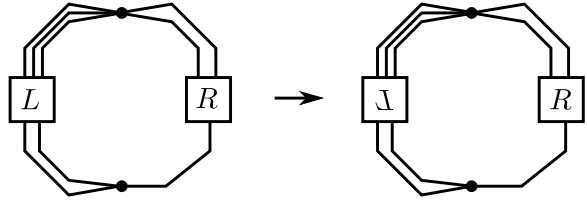}{3in}
$$
\caption{A Whitney flip on a graph.}
\lbl{f.flip}
\end{figure}

In \cite{GL1} it was shown that
\begin{itemize}
\item
a Whitney flip on a planar graph corresponds to a Conway mutation for the
corresponding alternating links.
\item
A flype move can be obtained by two Whitney flip
moves.
\end{itemize}
Menasco~\cite{Menasco} shows that there are two types of
Conway mutation, type I (visible in an alternating link projection) and type II
(hidden from the link projection). It was pointed out to us by F. Bonahon and
J. Greene that a type II mutation can be achieved by two type I mutations.
Independent of this fact, in \cite[Thm.1.1]{Greene} Greene proves that the
Tait graph gives a 1-1 correspondence between the set of alternating links,
modulo Conway mutation and the set of planar graphs modulo flips.
A Conway mutation does not change the colored Jones polynomial, hence
$\Phi_G(q)$ does not change under Whitney flips on $G$.

Let $\calG^{\fl}$ denote the set of equivalence classes on $\calG$ induced
by the Whitney flip equivalence relation. Let $\calP^{\fl}$ denote the
subalgebra of $\calP$ that consists of all polynomials $P: \calG \to \BQ$
(where $P \in \calP$) that satisfy $P(G)=P(G')$ whenever $G$ and $G'$ are
related by a Whitney flip.

The above discussion gives rise to a map
\begin{equation}
\lbl{eq.Palt}
\calP^{\fl} \times \{\text{Alternating links}\}/\{\text{Conway mutation}\}
\longto \BQ
\end{equation}

\begin{proposition}
\lbl{prop.P2}
\rm{(a)} If $H$ is 2-edge-connected and isomorphic to every one of its 
Whitney flips, then $[H] \in \calP^{\fl}$. 
\newline
\rm{(b)} In particular,
$c_{41}, c_{42}, c_n \in \calP^{\fl}$ where $c_n$ is the $n$-cycle
and $c_{5,i} \in \calP^{\fl}$ for $i=1,\dots,5$ and
$c_{6,i} \in \calP^{\fl}$ for $i=1,\dots,19$.
\end{proposition}
It follows that each term in the right hand side of Equations 
\eqref{eq.C4}-\eqref{eq.C5} is an alternating link invariant.

\begin{proof}
For part (a), fix $H$ 2-edge connected and isomorphic to every one of its
Whitney flips, and let $G=A \cup B$ be a graph and $G'=A \cup B'$ a Whitney 
flip of $G$. If $\phi: H \to G'$ is an embedding, then we can write
$H=H_A \cup H_B$ with embeddings $H_A \to A$ and $H_B \to B$. 
If $H_A$ or $H_B$ is empty, then we can construct an embedding $\psi: H \to G$. 
Otherwise, we can construct an embedding $\psi: H' \to G$ where $H'$ is a 
Whitney flip of $H$. Since $H'$ is isomorphic to $H$, this gives an embedding
$\psi: H \to G$. It is easy to see that the map $\phi \mapsto \psi$ is
a bijection, hence the number of embeddings of $H$ in $G$ is equal to the
number of embeddings of $H$ in $G'$.

For part (b), 
since Whitney flips preserve the number of vertices, by Proposition
\ref{prop.P2} it suffices to show that no two of the graphs $Gv_i^5$
(and similarly $Gv_j^6$) differ by Whitney flips.
In \cite[Sec.13.2]{GL1} it was shown that if  two planar graphs differ
by Whitney flips, the corresponding alternating links are Conway mutant, and
hence they have equal colored Jones polynomial, hence equal $\Phi(q)$
invariant. Inspection shows that the 5 irreducible graphs with 5 vertices
shown in Figure \ref{f.Gv5} and the 19 irreducible graphs with 6 vertices
shown in Figure \ref{f.Gv6} all have different Jones polynomial.
Therefore, no two graphs are flip equivalent.
\end{proof}

Let
$$
(\gamma,\delta)=( [ \psdraw{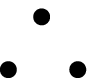}{0.3in}], [ \psdraw{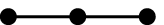}{0.5in}])
$$

\begin{lemma}
\lbl{lem.gamma.delta}
(a) $\gamma-\delta=
\frac{1}{6}[\bullet]^3
+2[\psdraw{myG4}{0.3in}]
-[\psdraw{myG2}{0.3in}][\bullet]+2[\psdraw{myG2}{0.3in}]
-\frac{1}{2}[\bullet]^2+\frac{1}{3}[\bullet]$.\\
(b) $\gamma-\delta\in\mathcal{P}^{\text{fl}}$ is an invariant of alternating 
links, polynomially determined by $c_1,c_2,c_3$.
\end{lemma}

\begin{proof}
(a) By the multiplication formula \eqref{eq.mult} we have
\begin{align}
[\bullet][\bullet] &= 
2[\bullet~~\bullet]
+2[\psdraw{myG2}{0.3in}]
+[\bullet] \lbl{A1}\\
[\bullet \,\, \bullet][\bullet] &=
2[\bullet~~\bullet] 
+[\psdraw{myG7}{0.3in}]
+2[\psdraw{myG2}{0.3in}~~\bullet]
+3[\psdraw{myG6}{0.3in}] \lbl{A2}\\
[\psdraw{myG2}{0.3in}][\bullet] &=
2[\psdraw{myG2}{0.3in}]
+3[\psdraw{myG4}{0.3in}]
+2[\psdraw{myG7}{0.3in}]
+ [\psdraw{myG2}{0.3in}~~\bullet] \lbl{A3}
\end{align}
It follows that
\begin{align}
[\bullet]^3 &=
2[\bullet~~\bullet][\bullet]
+2[\psdraw{myG2}{0.3in}][\bullet]
+[\bullet][\bullet] \notag  \\
&= 
2\,(2[\bullet~~\bullet] 
+[\psdraw{myG7}{0.3in}]+
2[\psdraw{myG2}{0.3in}~~\bullet]
+3[\psdraw{myG6}{0.3in}]) \notag \\
&  
+ 2\,(2[\psdraw{myG2}{0.3in}]
+3[\psdraw{myG4}{0.3in}]
+2[\psdraw{myG7}{0.3in}]
+[\psdraw{myG2}{0.3in}~~\bullet]) \notag \\
& 
+ 2[\bullet~~\bullet]
+2[\psdraw{myG2}{0.3in}]
+[\bullet] \notag \\[0pt]
&= 
6[\psdraw{myG4}{0.3in}]
+6[\psdraw{myG2}{0.3in}~~\bullet] 
+6[\psdraw{myG6}{0.3in}]
+6[\psdraw{myG7}{0.3in}]
+6[\bullet~~\bullet]
+6[\psdraw{myG2}{0.3in}]+[\bullet] \notag
\end{align}
Therefore
\begin{equation}
\lbl{A4}
6[\psdraw{myG6}{0.3in}] 
=[\bullet]^3
-6[\psdraw{myG4}{0.3in}]
-6[\psdraw{myG2}{0.3in}~~\bullet] 
-6[\psdraw{myG7}{0.3in}]
-6[\bullet~~\bullet]
-6[\psdraw{myG2}{0.3in}]-[\bullet]
\end{equation}
On the other hand, from Equation \eqref{A1} we have
\begin{equation}
\lbl{A5}
[\bullet~~\bullet]=\frac{1}{2}[\bullet] ^2-[\psdraw{myG2}{0.3in}]
-\frac{1}{2}[\bullet]
\end{equation}
and from Equation \eqref{A3}
\begin{equation}
\lbl{A6}
[\psdraw{myG2}{0.3in}~~\bullet]
=[\psdraw{myG2}{0.3in}][\bullet]
- 2[\psdraw{myG2}{0.3in}]
-3[\psdraw{myG4}{0.3in}]
-2[\psdraw{myG7}{0.3in}]
\end{equation}
Equations \eqref{A4},\eqref{A5},\eqref{A6} give
\begin{align}
6[\psdraw{myG6}{0.3in}] &= 
[\bullet]^3
-6[\psdraw{myG4}{0.3in}]
-6([\psdraw{myG2}{0.3in}][\bullet]
- 2[\psdraw{myG2}{0.3in}]
-3[\psdraw{myG4}{0.3in}]
-2[\psdraw{myG7}{0.3in}]) \notag \\
& - 6[\psdraw{myG7}{0.3in}]
-6(\frac{1}{2}[\bullet] ^2
-[\psdraw{myG2}{0.3in}]
-\frac{1}{2}[\bullet]
)-6[\psdraw{myG2}{0.3in}]
-[\bullet] \notag \\
&= [\bullet]^3
+12[\psdraw{myG4}{0.3in}]
-6[\psdraw{myG2}{0.3in}][\bullet]
+12[\psdraw{myG2}{0.3in}]
-3[\bullet]^2
+2[\bullet]
+6[\psdraw{myG7}{0.3in}] \notag
\end{align}
So
\begin{equation}
[\psdraw{myG6}{0.3in}]-[\psdraw{myG7}{0.3in}]=
\frac{1}{6}[\bullet]^3+2[\psdraw{myG4}{0.3in}]
-[\psdraw{myG2}{0.3in}][\bullet]+2[\psdraw{myG2}{0.3in}]
-\frac{1}{2}[\bullet]^2+\frac{1}{3}[\bullet] \notag
\end{equation}
(b) This follows from (a) and Proposition \ref{prop.P2}.
\end{proof}

Define the $k$-th moment of a graph to be the sum of the $k$-powers 
of the degrees (i.e., valencies) of the vertices of $G$.
The next lemma shows that the second moment is a polynomial of
an induced graph counting problem. This holds for all moments,
though we do not need this more general statement.

\begin{lemma}
\lbl{lem.pairs}
We have:
\begin{align} 
\sum_{v} \deg(v) &= 2 [\psdraw{myG2}{0.3in}]
\\
\sum_{v} \binom{\deg(v)}{2} &= 
[\psdraw{myG7}{0.3in}] + 3 [\psdraw{myG4}{0.3in}] 
\end{align}
\end{lemma}

\begin{proof}
The equalities follow by a simple counting argument. For the first
equation every edge has two vertices. For the second equation, 
given a vertex $v$ of $G$ and
an unorderd pair of two distinct neighboring vertices $w,w'$ of $v$, either 
$ww'$ is an edge of $G$ (hence $vww'$ is an induced triangle and contributes
three times on the second moment) or not (and contributes $\delta$ to the 
second moment).
\end{proof}


\section{A review of the $q$-series $\Phi_G(q)$}
\lbl{sec.review}

\subsection{The $q$-series $\Phi_G(q)$}
\lbl{sub.phiGq}

In this section we will review the definition of the $q$-series $\Phi_G(q)$
of \cite{GL1} following our earlier work \cite{GV}. Fix a rooted plane 
multigraph $G$, i.e., a planar multigraph (possibly with loops and 
multiple edges) together with a drawing on the plane together with
a vertex $v_\infty$ of its unbounded face $p_\infty$. A {\em corner} $(p,v)$
of $G$ is a face $p$ of $G$ and a vertex $v$ of $p$.
An \emph{admissible state} $(a,b)$ of $G$ is an integer assignment 
$a_p$ for each face $p$ and $b_v$ for each vertex $v$ of $p$ such that 
\begin{itemize}
\item
$a_p+b_v\geq 0$ for all corners $(p,v)$ of $G$.
\item
For the unbounded face $p_\infty$ we have $a_\infty=0$.
\item
For the vertex $v_\infty$ of $G$ we have $b_{v_\infty}=0$.
\end{itemize}
In the formulas below, $v,w$ will denote vertices of $G$ and
$p$ a face of $G$. We also write $vw\in p$ if $v$,$w$ are vertices and 
$vw$ is an edge of $p$. 

For an admissible state $(a,b)$ and a face $p$ of $G$ with $l(p)$ edges, 
we define
\begin{align*}
\gamma(p) &=l(p)a_p^2+2a_p(b_1+b_2+ \dots +b_{l(p)}) 
\end{align*}
where $b_1,\dots,b_{l(p)}$ are the values of the state on the vertices
of $p$ in counterclockwise order. Let
\begin{equation}
\lbl{def.A}
A(a,b)= \left( \sum\limits_{p}\gamma(p) \right) 
+ 2 \left( \sum\limits_{e=(v_iv_j)}b_{v_i}b_{v_j}  \right) 
\end{equation}
where the first sum is over the set of all faces $p$ of $G$ (including
the unbounded one) the second sum is over the set of edges of $G$. Let
\begin{equation}
\lbl{def.B}
B(a,b)=2\sum\limits_{v}b_v+\sum\limits_{p}(l(p)-2) a_p
\end{equation}
where the $v$-summation is over the set of vertices of $G$ and the 
$p$-summation is over the set of all faces of $G$. This explains the notation
of Equation~\eqref{eq.defphi}.

For an admissible state $(a,b)$ and a face $p$ of $G$, let 
$b_p=\min\{b_v:v \in p\}$.

\begin{theorem}
\lbl{thm.2}\cite{GV}
\rm{(a)} We have
\begin{align}
\lbl{eq.Q2}
A(a,b)& =  \sum\limits_{p} \left(
l(p) (a_p+b_p)^2+2 (a_p+b_p)
\left(\sum\limits_{v \in p}(b_v-b_p)\right)
\right.\\ \notag & \left.
+ \sum\limits_{vv'\in p}(b_v-b_p)(b_{v'}-b_p) \right)
+  \sum\limits_{vv'\in p_\infty}b_vb_{v'} \,,
\end{align}
where the $p$-summation is over the set of all faces of $G$.
Each term in the above sum is manifestly nonnegative.
\newline
\rm{(b)} $B(a,b)$ can also be written as a finite sum of manifestly
nonnegative linear forms on $(a,b)$.
\newline
\rm{(c)}
If $\frac{1}{2}(A(a,b)+B(a,b))\leq N$ for some natural number $N$, then for
every $i$ and every $j$ there exist $c_i, c'_i$ and $c_j,c'_j$
(computed effectively from $G$) such that
$$
c_i N  \leq b_i  \leq c_i' N, \qquad\qquad
c_j' \sqrt{N}  \leq a_j  \leq c_j N  + c_j' \sqrt{N} \,.
$$
\end{theorem}

\subsection{Some properties of $\Phi_G(q)$}
\lbl{sub.Phiprop}

In this section we summarize some properties of $\Phi_G(q)$.

\begin{lemma}
\cite{AD,GL1}
\lbl{lem.reduced}
\rm{(a)} The series $\Phi_G(q)$ depends only on the abstract planar graph $G$
and not on its plane embedding, nor on the choice of vertex of the unbounded
face.
\newline
\rm{(b)} If $G=G_1 \sqcup G_2$ is disconnected, then
$$
(1-q) \Phi_{G}(q)= \Phi_{G_1}(q) \Phi_{G_2}(q) \,.
$$
\newline
\rm{(c)} If $G$ has a separating vertex $v$
and $G\setminus\{v\}=G_1 \sqcup G_2$, then
$$
\Phi_G(q)= \Phi_{G_1}(q) \Phi_{G_2}(q) \,.
$$
\newline
\rm{(d)}
If $G$ is a planar graph (possibly with multiple edges and loops) and
$G^\text{red}$ denotes the corresponding simple graph obtained by removing all 
loops and replacing all edges of multiplicity more than with edges of 
multiplicity one, then
$$
\Phi_G(q)=\Phi_{G^\text{red}}(q) \,.
$$
\end{lemma}
Note that we use the normalization
\begin{equation}
\lbl{eq.Phi.init}
\Phi_{\bullet}(q)= \Phi_{\psdraw{myG2}{0.3in}}(q)=1 \,.
\end{equation}

In view of the above lemma, in the rest of the paper $G$ will denote a 
simple, 2-edge-connected rooted plane graph.

\subsection{Some lemmas from \cite{GV}}
\lbl{sub.lemmas}

In this section we review the statements of some lemmas from \cite{GV}
which we use for the proof of Theorem \ref{thm.1}.

\begin{lemma}
\lbl{cor.2}\cite[Cor.3.2]{GV}
For a pair $(p,v)$  a 2-edge-connected graph $G$ where $p$ is a face and 
$v$ is a vertex  of $p$ we have $B(a,b)\geq a_{p}+b_{v}$.
\end{lemma}

The proofs of the three lemmas below can be found in \cite[Sec.4]{GV}.
\begin{lemma}
\lbl{lem.triangle}
Let $G$ be a $2$-connected planar graph whose unbounded face has $V_\infty$
vertices. If $(a,b)$ is an admissible state such that
\begin{enumerate}
\item $b_v=b_{v'}=1$ where $vv'$ is an edge of $p_\infty$,
\item $a_p+b_p=0$ for any face $p$ of $G$,
\item $(b_{v_1}-b_p)(b_{v_2}-b_p)=0$ for any face $p$ of $G$ and edge $v_1v_2$
of $p$,
\end{enumerate}
then $b_v \geq 1$ for all vertices $v$, $a_p=-1$ for all faces
$p\neq p_{\infty}$ and $B(a,b)\geq 2+V_\infty$.
\end{lemma}

\begin{lemma}
\lbl{lem.triangle2}
Let $G$ be a $2$-connected planar graph whose unbounded face has
$V_\infty$ vertices. If $(a,b)$ is an admissible state such that
\begin{enumerate}
\item $b_v=b_{v'}=0$ and $(b_v-b_p)(b_{v'}-b_p)=1$ where $p$ is a boundary
face and $vv'$ is a boundary
 edge that belongs to $p$,
\item $a_p+b_p=0$ for any face $p$ of $G$,
\item $(b_{v_1}-b_p)(b_{v_2}-b_p)=0$ for any face $p$ of $G$ and edge
$v_1v_2$ not
on the boundary of $p$.
\end{enumerate}
Then $b_w \geq -1$ for all vertices $w$, $a_p=1$ for all
faces $p\neq p_{\infty}$ and $B(a,b)\geq V_\infty-2$.  Furthermore $B(a,b) =
 V_\infty-2$ if and only if
\begin{itemize}
\item $b_v=0$ for all boundary vertices $v$ and $b_w=-1$ for all other 
vertices $w$.
\item $a_p=1$ for all faces $p$.
\end{itemize}
\end{lemma}

\begin{lemma}
\lbl{lem.induction}
Let $G$ be a  $2$-connected planar graph, $p_0$ be a boundary face and
$(a,b)$ be an admissible state such that
\begin{enumerate}
\item
$a_{p_0}+b_{p_0}=0$,
\item
There exists a boundary edge  $vv'$ of $p_0$ such that $b_vb_{v'}=0$ and
$(b_{v}-b_{p_0})(b_{v'}-b_{p_0})=0$.
\end{enumerate}
Let $G_0$ be the graph obtained from $G$ by deleting the boundary edges of
$p_0$ and let $(a_0,b_0)$ be the restriction of the admissible state $(a,b)$
on $G_0$. Then,
\begin{itemize}
\item[(a)]
$(a_0,b_0)$ is an admissible state for $G_0$,
\item[(b)]
$A(a_0,b_0)=A(a,b)-\sum\limits_{e=(vv'):v,v'\in p_0\cap p_\infty}b_vb_{v'}$,
\item[(c)]
$B(a_0,b_0)=B(a,b)-2\sum\limits_{v\in V_0}b_v$, where $V_0$ is the set of
boundary vertices of $p_0$ that do not belong to any other bounded face,
\item[(d)]
$B(a,b)\geq 2\sum\limits_{v\in V_0}b_v$,
\item[(e)]
If furthermore $B(a,b)\leq 1$ then $A(a,b)=A(a_0,b_0), B(a,b)=B(a_0,b_0)$.
\end{itemize}
\end{lemma}


\section{The coefficient $q^3$ in $\Phi_G(q)$}
\lbl{sec.thm1}

\subsection{Analysis of admissible states}
\lbl{coef.q3}

In this section we find the admissible states $(a,b)$ such that
$\frac{1}{2}(A(a,b)+B(a,b))=3$. Since $A(a,b),B(a,b)\in \mathbb{N}$
we have the following cases:
$$
\begin{array}{|c|c|c|c|c|c|c|c|c|}
\hline
A(a,b) & 6 & 5 & 4 & 3 & 2 & 1 & 0 \\ \hline
B(a,b) & 0 & 1 & 2 & 3 & 4 & 5 & 6 \\ \hline
\end{array}
$$

{\bf Case 1}: $(A(a,b), B(a,b))=(6,0)$. By Lemma \ref{cor.2} we have
$B(a,b)\geq a_{p}+b_{p}\geq 0$ and so $a_p+b_p=0$ for all faces $p$.
Similarly since $B(a,b)\geq a_p+b_v= b_v-b_p\geq 0$ we have
$a_p+b_v=b_v-b_p=0$ for all $v\in p$. Thus $A(a,b)=6$ is equivalent to
\begin{equation}
\lbl{boundary4}
\sum\limits_{vv'\in p_\infty}b_vb_{v'}=6
\end{equation}
If $vv'$ is an edge of $G$ and $p$ is a face that contains $vv'$ then we
have $b_v=b_p=b_{v'}$. So by Equation
\eqref{boundary4} there exists a boundary edge $vv'$ such that $b_{v}=b_{v'}=1$.
Lemma \ref{lem.triangle} implies that $B(a,b)\geq 2+V_\infty>0$ which is
impossible. Therefore there are no admissible states $(a,b)$ that
satisfy $(A(a,b), B(a,b))=(6,0)$.\\

{\bf Case 2}: $(A(a,b), B(a,b))=(5,1)$. Since $l(p)\geq 3$ we have 
$a_p+b_p\leq 1$ for all $p$.

{\bf Case 2.1}: There exists a face $p_0$ such that $a_{p_0}+b_{p_0}=1$, which 
implies that $a_p+b_p=0$ for all $p\neq p_0$.

{\bf Case 2.1.1}: $l(p_0)=4$ or $5$. We have 
$B(a,b) \geq (a_{p_0}+b_{v_1})+(a_{p_0}+b_{v_2})=2(a_{p_0}+b_{p_0})=2$ 
which is impossible, here $v_1,v_2$ are two vertices of $p_0$.

{\bf Case 2.1.2}: $l(p_0)=3$. We have
$$
5=A(a,b)=3+\sum\limits_p \sum\limits_{vv'\in p}(b_v-b_p)(b_{v'}-b_p)
+  \sum\limits_{vv'\in p_\infty}b_vb_{v'}
$$
and therefore
\begin{equation}
\lbl{eq.212}
\sum\limits_p \sum\limits_{vv'\in p}(b_v-b_p)(b_{v'}-b_p)
+  \sum\limits_{vv'\in p_\infty}b_vb_{v'}=2
\end{equation}

There are at most two positive terms in Equation \eqref{eq.212}.  
Let $v_iv_i'\in p_i,~1\leq i\leq 2$  be the edges and bounded faces that 
appear in these terms. If a bounded face $p$ contains a boundary edge 
$vv'\neq v_iv_i',~i=1,2$ then we should have $b_vb_{v'}=(b_v-b_p)(b_{v'}-b_p)=0$.
This implies that $b_p=0$ and hence $a_p=0$.
Let $G'$ be the graph obtained from $G$ by deleting the boundary edges of
$p$ and $(a',b')$ be the restriction of $(a,b)$ on $G'$. By part (e) of Lemma
\ref{lem.induction} we have $A(a',b')=A(a,b)$, $B(a',b')=B(a,b)$. Continue 
this way until $G$ does not have any face $p$ with a boundary edge 
$vv'\neq v_iv_i'$, $i=1,2$. It is easy to see that the only possibility for 
this to happen is when $G=p_0\cup p_1\cup p_2$, where say $v_iv_i'\in p_i$, 
$i=1,2$. Since $p_1,p_2$ do not contain any boundary edge other than 
$v_iv_i',~i=1,2$, $G$ should be isomorphic to the graph in the following 
figure.
$$
\includegraphics[scale=1.3]{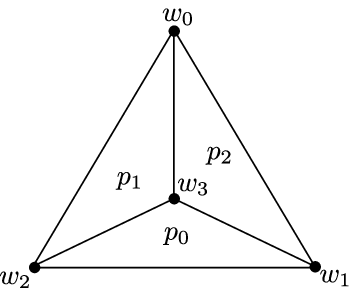}
$$
where $w_0w_1=v_1v_1',~w_0w_2=v_2v_2'$. It follows that $b_{w_1}b_{w_2}=0$ and 
let us assume that $b_{w_2}=0$, and so $b_{w_2}b_{w_0}=0$. This forces 
$(b_{w_2}-b_{p_1})(b_{w_0}-b_{p_1})=1$ since the edge $w_0w_2$ corresponds to 
a positive term in Equation \eqref{eq.212}, which must equal 1. It follows 
from the latter that $b_{w_0}=b_{w_2}=0$ and therefore from Equation 
\eqref{eq.212},  
$2=\sum\limits_{vv'\in p_\infty}b_vb_{v'}=
b_{w_0}b_{w_1}+b_{w_1}b_{w_2}+b_{w_2}b_{w_0}=0$ which is impossible.

{\bf Case 2.2}: $a_p+b_p=0$ for all $p$. Then we have
\begin{equation}
\lbl{eq.22}
\sum\limits_p \sum\limits_{vv'\in p}(b_v-b_p)(b_{v'}-b_p)
+  \sum\limits_{vv'\in p_\infty}b_vb_{v'}=5
\end{equation}
There are at most $5$ positive terms in Equation \eqref{eq.22}. Let 
$v_iv_i'\in p_i,~1\leq i\leq 5$  be the edges and bounded faces that appear 
in these terms. If a bounded face $p$ contains a boundary edge 
$vv'\neq v_iv_i',~1\leq i\leq 5$ then we should have  
$b_vb_{v'}=(b_v-b_p)(b_{v'}-b_p)=0$. This implies that $b_p=0$ and hence 
$a_p=0$.  Let $G'$ be the graph obtained from $G$ by deleting the boundary 
edges of $p$ and $(a',b')$ be the restriction of $(a,b)$ on $G'$. By part 
(e) of Lemma \ref{lem.induction} we have $A(a',b')=A(a,b)$, $B(a',b')=B(a,b)$. 
We can continue this way until all the boundary edges of $G$ are among the 
$v_iv_i'$. This means we can assume that $G$ has $m$ boundary edges where 
$3\leq m \leq 5$. Let us relabel the boundary vertices by   $v_1,v_2,..,v_m$.

{\bf Case 2.2.1}: All the positive terms in Equation \eqref{eq.22} 
correspond to boundary edges. If the positive terms are 
$b_{v_1}b_{v_{2}},...,b_{v_m}b_{v_1}$ then since 
$b_{v_1}b_{v_{2}}+ ... + b_{v_m}b_{v_1}=5$,
\begin{itemize}
\item there exists $1\leq i\leq m$ such that $b_{v_i}b_{v_{i+1}}=1$,
\item $(b_v-b_p)(b_{v'}-b_p)=0$ for all faces $p$ and edges $vv'$ of $G$.
\end{itemize}
It follows from \ref{lem.triangle} that $B(a,b)\geq V_\infty+2\geq 5$ 
which is impossible.
On the other hand, if for instance $b_{v_1}b_{v_2}=0$ then we can assume that 
$b_{v_1}=0$. Since the edge $v_1v_2$ corresponds to a positive term, we have
\begin{equation}
\lbl{eq.case221}
(b_{v_1}-b_{p_1})(b_{v_2}-b_{p_1})=k
\end{equation}
where $1\leq k\leq 3$ and $p_1$ is the bounded face that contains $v_1v_2$. 
Here $k\neq 4, 5$ since we are assuming that all positive terms correspond 
to boundary edges and there are at least 3 edges. We claim that $k=1$. 
Indeed, let us assume to the contradiction that $k\geq 2$. Equation 
\eqref{eq.case221} implies that either $b_{p_1}=-k$ and $b_{v_2}-b_{p_1}=1$ 
or $b_{p_1}=-1$ and $b_{v_2}-b_{p_1}=k$. The former is impossible since 
$b_{v_2}\geq 0$. From the later we have $b_{v_2}=k-1$ and  since 
$a_{p_1}+b_{p_1}=0$ we also have $a_{p_1}=1$. So by Lemma \ref{cor.2} we have 
$B(a,b)\geq a_{p_1}+b_{v_2}=k\geq 2$ which is impossible and the claim is 
proven. Therefore $k=1$ and hence $b_{v_1}=b_{v_2}=0$, $b_{p_1}=-1$. It 
follows that $b_{v_2}b_{v_3}=0$ which means 
$(b_{v_2}-b_{p_2})(b_{v_3}-b_{p_2})=k'$, $1\leq k'\leq 3$, because the edge 
$v_2v_3$ corresponds to a positive term. By a similar argument we can show 
that $k'=1$ and $b_{p_2}=-1$, $b_{v_3}=0$. Similarly we can prove that 
$b_{v_i}=0$ and $b_{p_i}=-1$ for all $1\leq i\leq 5$ for all $1\leq i\leq m$  
where $p_i$ is the boundary face that contains $v_iv_{i+1}$. In particular, 
this implies that $m=5$ and $(b_{v_i}-b_{p_i})(b_{v_{i+1}}-b_{p_i})=1$  for 
$1\leq i\leq 5$ and therefore $(b_v-b_p)(b_{v'}-b_p)=0$ for all 
$(p,vv')\neq (p_i,v_iv_{i+1})$ for all $i$.  So by Lemma \ref{lem.triangle2} 
we have $B(a,b)\geq V_\infty	-2=3$ which is impossible.

{\bf{Case 2.2.2}}: There are 1 or 2 positive terms in Equation  
\eqref{eq.22} that do not correspond to the boundary edges. By a similar 
argument as the above, we can reduce this to the case where the unbounded 
face of $G$ has 3 or 4 vertices. Let us consider the case where $G$ has 4 
boundary edges $v_1v_2,v_2v_3,v_3v_4,v_4v_1$ that correspond to $4$ of the 
$5$ positive terms and the other positive term corresponds to an edge 
$v_5v_6$ inside of $G$ as in the figure below. The other cases are 
completely similar.
$$
\psdraw{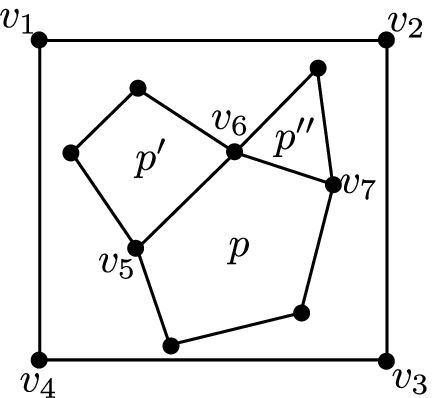}{1.5in}
$$
If the positive terms that correspond to the boundary edges are 
$b_{v_1}b_{v_{2}},...,b_{v_4}b_{v_1}$ then since 
$b_{v_1}b_{v_{2}}+ ... + b_{v_4}b_{v_1}=4$. This means that each of the terms 
$b_{v_i}b_{v_{i+1}}$ is equal
to $1$ and by an argument similar to the one of Case 2.2.1 we can conclude 
that $B(a,b)\geq V_\infty+2=6$, which is impossible. If, say 
$b_{v_1}b_{v_{2}}=0$ then $(b_{v_1}-b_{p_1})(b_{v_{2}}-b_{p_1})=k>0$ since the 
edge $v_1v_2$ corresponds to a
positive term, here $p_1$ is the bounded face that contains $v_1v_2$. 
Since we have
4 positive terms and 4 boundary edges, each positive term is equal to $1$, 
hence $k=1$.
Similar to the argument in Case 2.2.1, we can show that $b_p=-1$ for all 
faces $p$. Let $p$ be the face that appears in the positive term that 
contains $v_5v_6$ and $p'$ be the other face that contains $v_5v_6$. It 
follows from $(b_{v_5}-b_p)(b_{v_6}-b_p)=1$ that $b_{v_5}=b_{v_6}=b_{p}+1=0$. 
Since $(b_{v_5}-b_{p'})(b_{v_6}-b_{p'})=0$ we have $b_{p'}=0$ which is impossible.

{\bf Case 3}: $(A(a,b), B(a,b))=(4,2)$.

{\bf Case 3.1}: There exists a face $p_0$ such that $a_{p_0}+b_{p_0}=1$, which 
implies
that  $a_p+b_p=0$ for all $p\neq p_0$. Since $A(a,b)=4$ we have $l(p_0)\leq 4$.

{\bf Case 3.1.1}: $l(p_0)=4$.  By a similar argument
to the case 2 of Section 4.3 in \cite{GV} we can show that this gives us 
the following set of admissible states
$(a,b)$:
\begin{itemize}
\item $a_{p_0}=1$ for a square face $p_0$,  $a_{p}=0$ for $p\neq p_0$,
\item  $b_v=0$ for all vertices $v$,
\end{itemize}
The contribution of this state to $\Phi_G(q)$ is
$$
\frac{q^3}{(1-q)^{l(p_0)}}=\frac{q^3}{(1-q)^4}=q^3+O(q^4)
$$

{\bf Case 3.1.2}: $l(p_0)=3$. We have
\begin{equation}
\lbl{eq.312}
\sum\limits_p \sum\limits_{vv'\in p}(b_v-b_p)(b_{v'}-b_p)
+  \sum\limits_{vv'\in p_\infty}b_vb_{v'}=1
\end{equation}

There is exactly one positive term in Equation \eqref{eq.312}.  Let 
$vv'\in p$  be the edge and bounded face that appears in this term. If a 
bounded
face $p'$ contains a boundary edge $ww'\neq vv'$ then we should have
$b_wb_{w'}=(b_w-b_{p'})(b_{w'}-b_{p'})=0$. This implies that $b_{p'}=0$ and 
hence $a_{p'}=0$.
Let $G'$ be the graph obtained from $G$ by deleting the boundary edges of
$p'$ and $(a',b')$ be the restriction of $(a,b)$ on $G'$. By parts (c) and 
(d) of Lemma
\ref{lem.induction} we have $A(a',b')=A(a,b)$, $B(a',b')=B(a,b)-2k$, 
$k\in\{0,1\}$. Here
\begin{itemize}
\item $k=0$ if and only if $b_v=0$ for all removed vertices $v$,
\item $k=1$ if there exists a removed vertex $v$ such that $b_v=1$ and 
$b_w=0$ for all other removed vertices $w$.
\end{itemize}
 We can continue this way until $G=p_0$ if $p=p_0$ or $G=p\cup p_0$ if 
$p\neq p_0$. Let us consider first the case where $p=p_0$. Let the three 
vertices of $p_0$ be $v,v',v''$ and $b_{p_0}=b_v$. We have 
$2\geq B(a,b)=a_{p_0}+2(b_v+b_{v'}+b_{v''})
=(a_{p_0}+b_{p_0})+b_{p_0}+2(b_{v'}+b_{v''})=1+
b_{p_0}+2(b_{v'}+b_{v''})$. It follows that $1\geq b_{p_0}+2(b_{v'}+b_{v''})$ 
and hence $b_{p_0}=b_{v'}=b_{v''}=0$ since they are all non-negative. This 
implies that $a_{p_0}=1$ and so $A(a,b)=3a_{p_0}^2+2a_{p_0}(b_v+b_{v'}+b_{v''})=3$ 
which is impossible. If $p\neq p_0$ then there should exist an edge $v_0v_0'$ 
of $p_0$ that does not correspond to a positive term and hence 
$b_{v_0}b_{v_0'}=0$. It follows that $b_{p_0}=0$ and so   $a_{p_0}=1$. This 
forces $b_v=0$ for all $v\in p_0$ since otherwise 
$B(a,b)=a_{p_0}+2\sum\limits_{v\in p_)}b_v\geq 3$ which is impossible. 
Similarly there should exist an edge $ww'$ of $p$ such that $b_wb_{w'}=0$ 
which implies that $a_p=0$ and hence $b_p=0$. If $p$ and $p_0$ are disjoint 
then we have $2=B(a,b)=B^p(a,b)+B^{p_0}(a,b)=B^p(a,b)+1$ where $B^p(a,b)$ 
denotes the restriction of $B(a,b)$ on $p$. It follows that $B^p(a,b)=1$ 
and the argument in  Lemma \ref{lem.induction} implies that $b_v=0$ for all 
$v\in p$. This is impossible since it gives 
$B(a,b)=a_p+2\sum\limits_{v\in p}b_v=0$. So $p$ and $p_0$ are not disjoint. 
If $v$ is a vertex of both $p$ and $p_0$ then $b_v=0$ and therefore $b_p=0$ 
which implies that $a_p=0$ since $a_p+b_p=0$.
$$
\psdraw{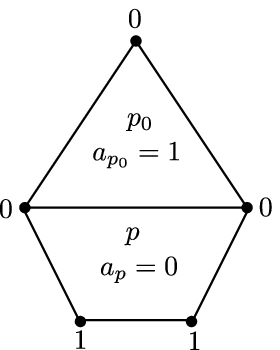}{1.3in}
$$
As before, the argument in  Lemma \ref{lem.induction} implies that $b_v=0$ 
for all $v\in p$ and so $B(a,b)=B^p(a,b)+B^{p_0}(a,b)=1$ which is impossible.

{\bf Case 3.2}: $a_p+b_p=0$ for all $p$. Then we have
\begin{equation}
\lbl{eq.32}
\sum\limits_p \sum\limits_{vv'\in p}(b_v-b_p)(b_{v'}-b_p)
+  \sum\limits_{vv'\in p_\infty}b_vb_{v'}=4
\end{equation}
There are at most $4$ positive terms in Equation \eqref{eq.32}. If an edge 
$vv'\in p$ does not correspond to a positive term then we should have
$b_vb_{v'}=(b_v-b_{p})(b_{v'}-b_{p})=0$. This implies that $b_{p}=0$ and hence 
$a_{p}=0$.
Let $G'$ be the graph obtained from $G$ by deleting the boundary edges of
$p$ and $(a',b')$ be the restriction of $(a,b)$ on $G'$. By parts (c) and 
(d) of Lemma
\ref{lem.induction} we have $A(a',b')=A(a,b)$, $B(a',b')=B(a,b)-2k$, 
$k\in\{0,1\}$. Here
\begin{itemize}
\item  $k=0$ if and only if $b_v=0$ for all removed vertices $v$,
\item $k=1$ if there exists a removed vertex $v$ such that $b_v=1$ and 
$b_w=0$ for all other removed vertices $w$.
\end{itemize}
 We can continue to do this until the boundary of $G$ has at most 4 edges 
all of which correspond to positive terms.

{\bf{Case 3.2.1}}:  All of the positive terms in Equation  \eqref{eq.32} 
correspond to boundary edges.

{\bf Case 3.2.1.1} $G$ has 3 vertices on the boundary, say $v_1,v_2,v_3$. 
If all the positive terms are equal to $1$ then there must exist a boundary 
edge, for instance, $v_1v_2$ of $G$ such that  
$b_{v_1}b_{v_2}=(b_{v_1}-b_{p_1})(b_{v_2}-b_{p_1})=1$ where $p_1$ is the bounded 
face that contains $v_1v_2$. This implies that $b_{p_1}=0$ and hence 
$a_{p_1}=0$. Let $vv'\not\in \{v_1v_2,v_2v_3,v_3v_1\}$ be another edge of 
$p_1$ and let $p$ be the other bounded face that contains $vv'$. Since 
$vv'$ does not correspond to a positive term, we have 
$(b_v-b_{p_1})(b_{v'}-b_{p_1})=0$ and so $b_vb_{v'}=0$. We also have 
$(b_v-b_{p})(b_{v'}-b_{p})=0$ which means $b_{p}=\min\{b_v,b_{v'}\}=0$ and 
hence $a_{p}=0$. Similarly we can show that $b_{p'}=a_{p'}=0$ for all faces 
$p'$ and in particular $b_w\geq 0$ for all $w$. It follows that 
$B(a,b)\geq 2(b_{v_1}+b_{v_2})=4$ which is impossible.

If one of the positive terms is equal to $2$ then the other two are equal 
to $1$. Without loss of generality we can assume that the edge $v_1v_2$ 
corresponds to this term, so either  
$b_{v_1}b_{v_2}=2$ or $(b_{v_1}-b_{p_1})(b_{v_2}-b_{p_2})=2$. For the former we 
can assume that $b_{v_1}=1$ and $b_{v_2}=2$. This implies that $b_{v_3}=0$ 
since otherwise $A(a,b)\geq b_{v_1}b_{v_2}+b_{v_2}b_{v_3}+b_{v_3}b_{v_1}
\geq 2+1+2=5$ which is impossible. Since $b_{v_2}b_{v_3}=0$ which means 
$(b_{v_2}-b_{p_2})(b_{v_3}-b_{p_2})=1$ and this leads to  $-b_{p_2}(2-b_{p_2})=1$ 
which is impossible.

{\bf Case 3.2.1.2}  $G$ has 4 vertices on the boundary, say 
$v_1,v_2,v_3,v_4$. By a similar argument to the case 2.2 of Section 4.3 in 
\cite{GV}, this corresponds to the following admissible state of $G$:
\begin{itemize}
\item $a_p=1$ for all bounded faces $p$,
\item $b_{v_1}=b_{v_2}=b_{v_3}=b_{v_4}=0$ where $v_1,v_2,v_3,v_4$ are the 
vertices of a square $G_0$ that does not have any diagonal in its interior.
 We will write $c_{40}=[G_0](G)$.
$$
G_0=\psdraw{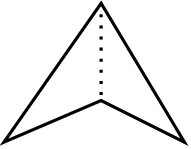}{1in}
$$
where the dotted line means $G_0$ does not contain an internal diagonal,
\item $b_w=-1$ for all vertices $w$ inside the 4-circle mentioned above,
\item $b_{\tilde{w}}=0$ for any other vertex $w$.
\end{itemize}

The contribution of this state to $\Phi_G(q)$ is
$$
\frac{q^3}{(1-q)^{\text{deg}_\square(v_1)+\text{deg}_\square(v_2)
+\text{deg}_\square(v_3)+\text{deg}_\square(v_4)-4}}=q^3+O(q^4)
$$
where $\text{deg}_\square(v)$ is the degree of $v$ in the square 
$\square=v_1v_2v_3v_4$.

{\bf{Case 3.2.2}}:  One of of the positive terms in Equation  
\eqref{eq.32} does not correspond to any boundary edge. By a similar 
argument to the Case 2.2.2 we can show that there are no admissible states 
here.

{\bf Case 4}: $(A(a,b), B(a,b))=(3,3)$. By a similar argument to the case 
2 of Section 4.3 in \cite{GV} we can show that the admissible states for 
this case are
\begin{itemize}
\item $a_p=1$ for all faces $p$.
\item $b_{v_1}=b_{v_2}=b_{v_3}=0$ where $v_1,v_2,v_3$ are the vertices of a 
3-cycle in $G$.
\item $b_v=-1$ for all $v$ inside the 3-cycle mentioned above.
\item $b_{v_0}=1$ for a fixed vertex $w$ outside of the 3-cycle.
\item $b_w=0$ for all other vertices $w$.
\end{itemize}
$$
\psdraw{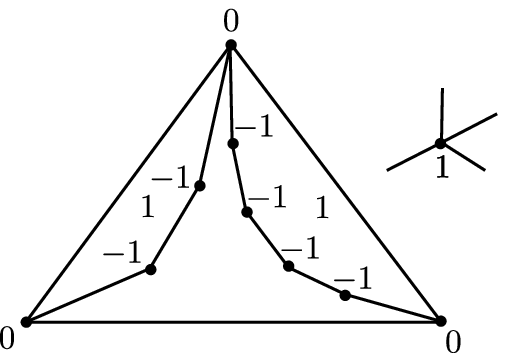}{2.0in}
$$
and
\begin{itemize}
\item $a_p=1$ for all faces $p$.
\item $b_{v_1}=b_{v_2}=b_{v_3}=0$ where $v_1,v_2,v_3$ are the vertices of a 
3-cycle in $G$.
\item $b_{v_0}=0$ for a fixed vertex $v_0$ inside the 3-cycle that is not 
adjacent to any of the vertices $v_1,v_2,v_3$ and $b_v=-1$ for all other 
$v$ also inside the cycle.
\item $b_w=0$ for all other vertices $w$.
\end{itemize}
$$
\psdraw{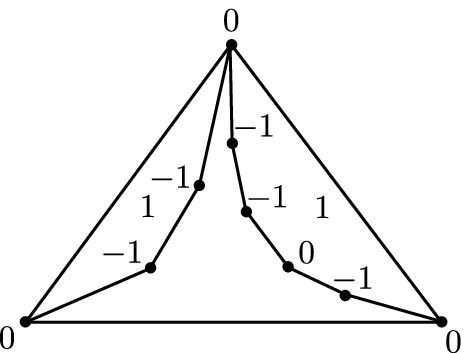}{2.0in}
$$
The contribution of both types of states above to $\Phi_G(q)$ is
$$
(-1)^3\frac{q^3}{(1-q)^{\text{deg}_\Delta(v_1)
+\text{deg}_\Delta(v_2)+\text{deg}_\Delta(v_3)+\text{deg}_\Delta(v_0)-3}}
=-q^3+O(q^4)
$$
where $\text{deg}_\Delta(v)$ is the degree of $v$ in the triangle 
$\Delta=v_1v_2v_3$.

{\bf Case 5}: $(A(a,b), B(a,b))=(2,4)$. Since $A(a,b)=2$, we have 
$a_p+b_p=0$ for all $p$ and
\begin{equation}
\lbl{eq.case5}
\sum\limits_p \sum\limits_{vv'\in p}(b_v-b_p)(b_{v'}-b_p)
+  \sum\limits_{vv'\in p_\infty}b_vb_{v'}=2
\end{equation}
There are at most $2$ positive terms in Equation \eqref{eq.case5}. If a 
boundary face $p$ contains a boundary edge $vv'$  that does not correspond 
to a positive term then we should have
$b_vb_{v'}=(b_v-b_{p})(b_{v'}-b_{p})=0$. This implies that $b_{p}=0$ and hence 
$a_{p}=0$.
Let $G'$ be the graph obtained from $G$ by deleting the boundary edges of
$p$ and $(a',b')$ be the restriction of $(a,b)$ on $G'$. By parts (c) and 
(d) of Lemma
\ref{lem.induction} we have $A(a',b')=A(a,b)-2i$, $B(a',b')=B(a,b)-2k$, 
$k\in\{0,1,2\}$. Here
\begin{itemize}
\item $i=0$ and $k=0$ if and only if $b_v=0$ for all removed vertices.
\item $i=0$ and $k=1$ if and only if there exists a removed vertex $v$ such 
that 
$b_v=1$ and $b_w=0$ for all other removed vertices $w$.
\item $i=0$ and $k=2$ if and only if there exist two removed vertices $v,v'$ 
which 
are not connected by an edge  such that $b_v=b_{v'}=1$ and  $b_w=0$ for all 
other removed vertices $w$.
\item $i=1$ and $k=2$ if and only if there exist two removed vertices 
$v,v'$ which 
are connected by an edge  such that $b_v=b_{v'}=1$ and  $b_w=0$ for all 
other removed vertices $w$.
\end{itemize}
It is easy to see that only the last item gives admissible states $(a,b)$ 
with $(A(a,b), B(a,b))=(2,4)$. To summarize, the admissible states in this 
case are those $(a,b)$ that satisfy
\begin{itemize}
\item $a_p=0$ for all faces $p$.
\item There exist two vertices $v,v'$ which are connected by an edge  such 
that $b_v=b_{v'}=1$ and  $b_w=0$ for all other vertices $w$.
\end{itemize}
The contribution of this state to $\Phi_G(q)$ is
$$
\frac{q^3}{(1-q)^{\text{deg}(v)+\text{deg}(v')}}=q^3+O(q^4)
$$

{\bf Case 6}: $(A(a,b), B(a,b))=(1,5)$. Since $A(a,b)=1$, we have 
$a_p+b_p=0$ for all $p$ and
\begin{equation}
\lbl{eq.case6}
\sum\limits_p \sum\limits_{vv'\in p}(b_v-b_p)(b_{v'}-b_p)
+  \sum\limits_{vv'\in p_\infty}b_vb_{v'}=1
\end{equation}
There is exactly $1$ positive term in Equation \eqref{eq.case6}. If the 
pair $(vv',p)$, $vv'\in p$ does not correspond to this positive term then 
we should have
$b_vb_{v'}=(b_v-b_{p})(b_{v'}-b_{p})=0$. This implies that $b_{p}=0$ and hence 
$a_{p}=0$. Similarly we can show that $a_{p'}=0$ for all other faces $p'$. 
This implies that $5=B(a,b)=2\sum\limits_{v}b_v$ which is impossible. So 
there are no admissible states in this case.

{\bf Case 7}: $(A(a,b), B(a,b))=(0,6)$. Since $A(a,b)=0$, we have 
$a_p+b_p=0$ for all $p$ and
\begin{equation}
\lbl{eq.case7}
\sum\limits_p \sum\limits_{vv'\in p}(b_v-b_p)(b_{v'}-b_p)
+  \sum\limits_{vv'\in p_\infty}b_vb_{v'}=0
\end{equation}
Let $vv'\in p$ where $p$ is a boundary face then we should have
$b_vb_{v'}=(b_v-b_{p})(b_{v'}-b_{p})=0$. This implies that $b_{p}=0$ and hence 
$a_{p}=0$.
Let $G'$ be the graph obtained from $G$ by deleting the boundary edges of
$p$ and $(a',b')$ be the restriction of $(a,b)$ on $G'$. By parts (c) and 
(d) of Lemma
\ref{lem.induction} we have $A(a',b')=A(a,b)$, $B(a',b')=B(a,b)-2k$, 
$k\in\{0,1,2,3\}$. Here
\begin{itemize}
\item $k=0$ if and only if $b_v=0$ for all removed vertices.
\item $k=1$ if and only if there exists a removed vertex $v$ such that 
$b_v=1$ and 
$b_w=0$ for all other removed vertices $w$.
\item $k=2$ if and only if there exists a removed vertex $v$ such that 
$b_v=2$ or if and only if 
there exist two removed vertices $v,v'$ which are not connected by an edge  
such that $b_v=b_{v'}=1$  and  $b_w=0$ for all other removed vertices $w$.
\item $k=3$  if and only if there exists a removed vertex $v$ such that 
$b_v=3$ or if and only if 
there exist two removed vertices $v,v'$ which are not connected by an edge  
such that $b_v=1,~b_{v'}=2$ or if and only if there exist three removed vertices 
$v,v',v''$ none of which are connected by an edge such that 
$b_v=b_{v'}=b_{v''}=1$   and  $b_w=0$ for all other removed vertices $w$.
\end{itemize}
The above possible values of $k$ lead to the following admissible states 
$(a,b)$:
\begin{itemize}
\item $a_p=0$ for all faces $p$.
\item There exists a vertex $v$ such that $b_v=3$ and  $b_w=0$ for all 
$w\neq v$.
\end{itemize}
The contribution of this state to $\Phi_G(q)$ is
$$
\frac{q^3}{(1-q)_3^{\text{deg}(v)}}=q^3+O(q^4)
$$
\begin{itemize}
\item $a_p=0$ for all faces $p$.
\item There exist two vertices $v,v'$ which are not connected by an edge  
such that $b_v=1,~b_{v'}=2$ and  $b_w=0$ for all other vertices $w$.
\end{itemize}
The contribution of this state to $\Phi_G(q)$ is
$$
\frac{q^3}{(1-q)^{\text{deg}(v)}(1-q)_2^{\text{deg}(v')}}=q^3+O(q^4)
$$
\begin{itemize}
\item $a_p=0$ for all faces $p$.
\item There exist three vertices $v,v',v''$ none of which are connected by 
an edge such that $b_v=b_{v'}=b_{v''}=1$  and  $b_w=0$ for all other vertices 
$w$.
\end{itemize}
The contribution of this state to $\Phi_G(q)$ is
$$
\frac{q^3}{(1-q)^{\text{deg}(v)+\text{deg}(v')+\text{deg}(v'')}}=q^3+O(q^4)
$$

\subsection{Proof of Theorem \ref{thm.1}}
\lbl{sub.proof.tm1}

We now give a proof of Theorem \ref{thm.1} based on cases 1-7 of Section
\ref{coef.q3}. We write
\begin{align*}
\Phi_G(q) &=(1-q)(q)_\infty^{c_2} (1+a_1 q+a_2 q^2 +a_3 q^3 +\text{O}(q^4))\\
&= (1-q)(1+b_1q+b_2q^2+b_3q^3)(1+a_1 q+a_2 q^2 +a_3 q^3) +\text{O}(q^4)
\end{align*}
where from \cite[Sec.{4.2}]{GV} we have
\begin{align*}
a_1 &= c_1 \\
a_2 &= \frac{c_1(c_1+1)}{2}+c_2-c_3
\end{align*}
and $a_3$ receives contributions from
\begin{itemize}
\item States $(a,b)$ such that  $\frac{1}{2}(A+B)=3$. These are discussed 
in Section \ref{coef.q3}.
\item States $(a,b)$ such that  $\frac{1}{2}(A+B)\leq 2$ which are discussed 
in \cite[Sec.{4.2}]{GV}.
\end{itemize}•
By expanding the factor $(q)_\infty^{c_2} $ we have
\begin{align*}
b_1 &=-c_2\\
b_2 &= \frac{c_2(c_2-3)}{2}\\
b_3 &= \frac{-c_2^3+9c_2^2-8c_2}{6}
\end{align*}
The total contribution of the admissible states found in cases 1-7 to 
$a_3q^3+O(q^4)$ is
\begin{equation}
\lbl{contribution.q3}
(c_{40}+c_1+c_2+2(\frac{c_1(c_1-1)}{2}-c_2)
+\gamma-\sum\limits_{C_3=vv'v''}(c_1-\alpha(C_3)))q^3 +\text{O}(q^4)
\end{equation}
where $\frac{c_1(c_1-1)}{2}-c_2$ is the number of pair of vertices in $G$ 
that are not connected by an edge. The last term is a summation over 
3-cycles $C_3=vv'v''$ of $G$ and $\alpha(C_3)$ is $3$ plus the number of 
vertices 
contained in $C_3$ that are adjacent to either $v,v'$ or $v''$. The 
admissible states in Sections 4.2 and 4.3 in \cite{GV} gives the following 
contribution to $a_3q^3+O(q^4)$:
\begin{align}
\lbl{contribution.q2}
& 1+\sum\limits_v q(1+q+q^2)^{\text{deg}(v)}
-q^2\sum\limits_{C_3=vv'v''}(1+q)^{\text{deg}_{C3}(v)
+\text{deg}_{C3}(v')+\text{deg}_{C3}(v'')-3}\\
& +q^2\sum\limits_{(vv')\neq e}(1+q)^{\text{deg}(v)+\text{deg}(v')}
+\sum\limits_{v}q^2(1+q)^{\text{deg}(v)} \notag
 +\text{O}(q^4)
\end{align}
where by $(vv')\neq e$ we mean a pair of vertices $v$, $v'$ that are not 
connected by an edge and
deg$_{C3}(v)$ denotes the degree of $v$ in the subgraph of $G$ that is 
contained in $C_3$.
Summing up \eqref{contribution.q3} and \eqref{contribution.q2} we get
\begin{align}
\lbl{contribution1.q3}
a_3 &= c_{40}+c_1+c_2+2(\frac{c_1(c_1-1)}{2}-c_2)+\gamma+\delta+c_2
+\sum\limits_{(vv')\neq e}(\text{deg}(v)+\text{deg}(v'))\\
&+2c_2 -\sum\limits_{C_3=vv'v''}(c_1+\text{deg}_{C3}(v)+\text{deg}_{C3}(v')
+\text{deg}_{C3}(v'')-3-\alpha(C_3)) \notag
\end{align}
Note that
\begin{align*} \sum\limits_{(vv')\neq e}(\text{deg}(v)+\text{deg}(v')) 
&= \sum\limits_{v}\text{deg}(v)(c_1-1-\text{deg}(v))\\
&= 2c_2(c_1-1)-\sum\limits_{v}(\text{deg}(v))^2\\
&= 2c_2(c_1-1)-2\delta -6 c_3 -2 c_2
\end{align*}
where the last equality follows from Lemma~\ref{lem.pairs}.
Let us define
$$
d_3=\text{deg}_{C3}(v)+\text{deg}_{C3}(v')+\text{deg}_{C3}(v'')-3-\alpha(C_3)
$$
and $c_{40}'= \lb \psdraw{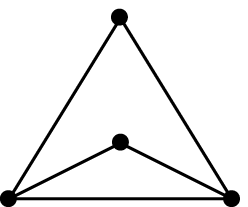}{0.2in} \rb, c_{41}= \lb \psdraw{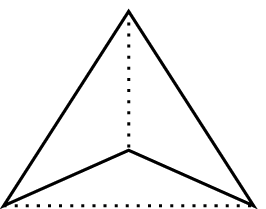}{0.2in} \rb$.
\begin{lemma}
\lbl{lem.d3}
We have\\
(a) $d_3=c_{40}'+2c_{42}$.\\
(b) $c_{40}-c_{40}'=c_{41}$.
\end{lemma}
\begin{proof}
(a) If $w$ is a vertex in the interior incident to $v$ and $v'$ then it 
contributes +1 to deg($v$), +1 to deg($v'$) and -1 to itself. Hence totally 
such $w$'s contribute $c_{40}'$.
If $w$ is a vertex in the interior incident to $v$, $v'$, $v''$ then it 
contributes +1 to each deg($v$),  deg($v'$), deg($v''$) and -1 to itself. 
So totally such $w$'s contribute $2c_{42}$. If $w$ is a boundary vertex 
then its contribution to each of deg($v$),  deg($v'$), deg($v''$) is +2 
and the total contribution of the 3 boundary vertices is +6 which cancels 
the -6 in $d_3$. Thus we have
$$
d_3=c_{40}'+ 2c_{42}
$$
(b) We have \begin{eqnarray*}
c_{40}-c_{40}' &=& \lb \psdraw{G1}{0.3in} \rb-\lb \psdraw{G11}{0.3in}\rb\\
&=& \lb \psdraw{G00}{0.3in} \rb\\
&=& c_{41}
\end{eqnarray*}
\end{proof}

Therefore Equation \eqref{contribution1.q3} combined with Lemmas 
\ref{lem.gamma.delta} and \ref{lem.d3} gives that
\begin{align*}
a_3 &= c_{41}-2c_{42}-c_3c_1+c_1+c_2+2(\frac{c_1(c_1-1)}{2}-c_2)+\gamma
+\delta+c_2+2c_2(c_1-1)-2\delta+2c_2\\
 &=2c_1c_2+c_1^2-c_3c_1+c_{41}-2c_{42}+\gamma-\delta\\
&=\frac{{c_1}^3}{6}+\frac{{c_1}^2}{2}+{c_1} {c_2}-{c_1}
   {c_3}+\frac{{c_1}}{3}+{c_2}-{c_3}+{c_{41}}-2 {{c_{42}}}
\end{align*}

Therefore the coefficient $\phi_{G,3}$ of $q^3$ in $\Phi_G(q)$ is given by
\begin{align*}
 \phi_{G,3} &=a_3+b_3+a_1b_2+a_2b_1-a_2-b_2-a_1b_1\\
&=c_{41} - 2 c_{42} + \frac{c_2}{6} + c_3 c_2 - \frac{c_2^3}{6}
- \frac{c_1}{6} - c_3 c_1 + \frac{c_2^2 c_1}{2} -
\frac{ c_2 c_1^2}{2} + \frac{c_1^3}{6}
\end{align*}
This completes the proof of Theorem \ref{thm.1}.
\qed

\subsection*{Acknowledgment}
S.G. was supported in part by a National Science Foundation
grant DMS-0805078. S.N. was supported by an NSERC Discovery grant.


\appendix

\section{Computations}
\lbl{sec.compute}

Tables \ref{f.compute} and \ref{f.Gvcompute} illustrate Theorem \ref{thm.1}
and confirm Conjecture \ref{conj.3} for all alternating links with at most
10 crossings and all irreducible planar graphs with at most 7 vertices.
These tables were compiled as follows.

\begin{itemize}
\item
We use {\tt Sage} to list all irreducible planar graphs with at most
10 edges (using the notation of \cite[App.A]{GV}).
\item
We use a {\tt Mathematica} program to compute the corresponding vectors
$c$ and $C$ and the series
$\Phi_G(q)+O(q^4)$ of Theorem \ref{thm.1}.
\item
To identify the corresponding alternating links $L$, we use a
{\tt Mathematica} program that converts the adjacency matrix of a planar
graph $G$ to the Dowker-Thistlethwaite code of the corresponding
alternating link $L$, and then use {\tt SnapPy} (see \cite{snappy}) to
identify the link with one of the Rolfsen's table \cite{Rf}
(if $L$ has at most 10 crossings) or Thislethwaite's table
(if $L$ has more than 10 crossings).
\item
We compute the stable coefficients $\Phi_L(q)+O(q^6)$ using
{\tt KnotAtlas} (see \cite{B-N}) which computes the colored Jones polynomials
of a link.
\end{itemize}
The equality $\Phi_G(q)=\Phi_L(q)$ of Theorem \ref{thm.1} is observed up to
$O(q^4)$ and Conjecture \ref{conj.3} is verified for all such graphs.

\begin{remark}
\lbl{rem.inequalities.ev}
If $G$ is a connected planar graph with $v$ vertices and $e$ edges, the
following inequalities bound $e$ in terms of $v$ and vice-versa
$$
v \leq e \qquad \text{and} \qquad e \leq 3v-6
$$
\end{remark}

\begin{figure}[!htpb]
$$
\begin{array}{|c|c|c|c|c|c|c|c|c|} \hline
\text{crossings}=\text{edges} & 3 & 4 & 5 & 6 & 7 & 8 & 9 & 10 \\ \hline
\text{alternating links} & 1 & 2 & 3 & 8 & 14 & 39 & 96 & 297 \\ \hline
\text{irreducible graphs} & 1 & 1 & 1 & 3 & 3 & 8 & 17 & 41 \\ \hline
\end{array}
$$
\caption{The number of alternating links with at most 10 crossings and
the number of irreducible graphs with at most 10 edges.}
\lbl{f.numbers}
\end{figure}

\begin{figure}[!htpb]
{\small
$$
\begin{array}{|c|c|l|l|l|} \hline
G & c & C & L & \Phi_L(q)+O(q^6)  \\ \hline
G^3_0 & 3, 3, 1, 0, 0  & 1, 1, 1, 1, 1 & 3_1 & 1 - q - q^2 + q^5   \\ \hline
\hline
G^4_0 & 4, 4, 0, 1, 0 & 1, 0, -1, -1, -1 & 4^2_1 & 1 - q + q^3  \\ \hline
\hline
G^5_0 & 5, 5, 0, 0, 0 & 1, 0, 0, 1, 1 & 5_1 & 1 - q - q^4   \\ \hline
\hline
G^6_0 & 6, 6, 0, 0, 0 & 1, 0, 0, 0, -1 & 6_1^2& 1 - q + q^5  \\ \hline
G^6_1 & 4, 6, 4, 0, 1 & 3, 4, 6, 9, 16 & 6^3_2 & 1 - 3 q - q^2 + 5 q^3 + 3 q^4 +3q^5\\ \hline
G^6_2 & 5, 6, 0, 3, 0 & 2, 0, -3, -4, -3 & 6^3_1 & 1 - 2 q + q^2 + 3 q^3 - 2 q^4 -2q^5 \\ \hline \hline
G^7_0 & 7, 7, 0, 0, 0 & 1, 0, 0, 0, 0 &7_1  & 1-q \\ \hline
G^7_1& 5, 7, 2, 2, 0 & 3, 2, 0, -2, -4 & 7^2_6 & 1 - 3 q + q^2 + 5 q^3 - 3 q^4 -3q^5\\ \hline
G^7_2  & 6, 7, 0, 1, 0  & 2, 0, -1, 1, 2 & 7^2_4 & 1 - 2 q + q^2 + q^3 - 3 q^4 +q^5 \\ \hline
\end{array}
$$
}
\caption{The irreducible graphs $G$ with at most 10 edges, the 6-tuple
of polynomial invariants $c=(c_1,c_2,c_3,c_{41},c_{42})$,
$C=(C_1,C_2,C_3,c_4,C_5)$ as defined in Equation \eqref{eq.phiL}, the
alternating link $L$ and the 6 stable coefficients of the Jones polynomial
of $L$.
}
\lbl{f.compute}
\end{figure}

\begin{figure}[!htpb]
{\small
$$
\begin{array}{|c|c|l|l|l|} \hline
G^8_0 & 8, 8, 0, 0, 0 & 1, 0, 0, 0, 0  & 8^2_1 & 1 -q \\ \hline
G^8_1 & 5, 8, 4, 1, 0 & 4, 4, 3, 0, -6 & 8_{18} & 1 - 4 q + 2 q^2 + 9 q^3 - 5 q^4 -8 q^5 \\ \hline
G^8_2 & 6, 8, 0, 5, 0 & 3, 0, -5, -7, -4 & 8^2_{14} & 1 - 3 q + 3 q^2 + 4 q^3 - 8 q^4-2q^5 \\ \hline
G^8_3 & 6, 8, 2, 0, 0 & 3, 2, 2, 4, 6 & 8^3_{15} & 1 - 3 q + q^2 + 3 q^3 - 3 q^4+3q^5 \\ \hline
G^8_4 & 6, 8, 1, 2, 0 & 3, 1, -1, 0, 2 & 8_{16} & 1 - 3 q + 2 q^2 + 3 q^3 - 6 q^4+q^5\\ \hline
G^8_5 & 6, 8, 0, 6, 0 & 3, 0, -6, -10, -7 & 8^4_{1} & 1 - 3 q + 3 q^2 + 5 q^3 - 8 q^4-5q^5 \\ \hline
G^8_6 & 7, 8, 0, 1, 0 & 2, 0, -1, -1, -3 & 8^3_{1} & 1 - 2 q + q^2 + q^3 - q^4+2q^5\\ \hline
G^8_7 & 7, 8, 0, 0, 0 & 2, 0, 0, 2, 1 & 8_5 & 1 - 2 q + q^2 - 2 q^4 +3q^5 \\ \hline
\end{array}
$$
}
\caption{Figure \ref{f.compute} continued.
}
\lbl{f.compute2}
\end{figure}

\begin{figure}[!htpb]
{\small
$$
\begin{array}{|c|c|l|l|l|} \hline
G^9_1 & 5, 9, 7, 0, 2 & 5, 7, 11, 17, 31 & 9_{40} & 1 - 5 q + 3 q^2 + 14 q^3 - 6q^4-15q^5 \\ \hline
G^9_2 &6, 9, 2, 5, 0 &  4, 2, -3, -9, -13 & 9^3_{12} & 1 - 4 q + 4 q^2 + 7 q^3 - 13q^4 -7q^5 \\ \hline
G^9_3 & 6, 9, 3, 1, 0 & 4, 3, 2, 3, 6 & 9^2_{42} & 1 - 4 q + 3 q^2 + 6 q^3 - 9q^4 \\ \hline
G^9_4 & 6, 9, 2, 4, 0 & 4, 2, -2, -5, -5 & 9_{34} & 1 - 4 q + 4 q^2 + 6 q^3 - 13 q^4-3q^5 \\ \hline
G^9_5 & 6, 9, 2, 3, 0 & 4, 2, -1, -1, 3 & 9_{40} & 1 - 4 q + 4 q^2 + 5 q^3 - 13 q^4+q^5 \\ \hline
G^9_6 & 7, 9, 0, 3, 0 & 3, 0, -3, -3, -4 & 9^2_{40} & 1 - 3 q + 3 q^2 + 2 q^3 - 6 q^4 +4q^5\\ \hline
G^9_7 & 7, 9, 1, 0, 0 & 3, 0, -3, -3, -4 & 9_{41} & 1 - 3 q + 2 q^2 + q^3 - 4 q^4 +7q^5\\ \hline
G^9_8  & 7, 9, 2, 0, 0 & 3, 2, 2, 2, 0 & 9^2_{31} & 1 - 3 q + q^2 + 3 q^3 - q^4 +3q^5 \\ \hline
G^9_{9} & 7, 9, 0, 3, 0 & 3, 0, -3, -2, 0 & 9^2_{36} & 1 - 3 q + 3 q^2 + 2 q^3 - 7q^4 +3q^5 \\ \hline
G^9_{10}  & 7, 9, 1, 1, 0 & 3, 1, 0, 1, 0 & 9^2_{35} & 1 - 3 q + 2 q^2 + 2 q^3 - 4q^4 +4q^5 \\ \hline
G^9_{11} & 7, 9, 0, 2, 0 & 3, 0, -2, 1, 3 & 9_{29} & 1 - 3 q + 3 q^2 + q^3 - 7 q^4+6q^5 \\ \hline
G^9_{12} & 7, 9, 0, 3, 0 & 3, 0, -3, -1, 3 & 9^3_{3} & 1 - 3 q + 3 q^2 + 2 q^3 - 8q^4+3q^5 \\ \hline
G^9_{13}  & 7, 9, 0, 2, 0 & 3, 0, -2, 2, 6 & 9^3_{9} & 1 - 3q + 3q^2 + q^3 - 8q^4+6q^5 \\ \hline
G^9_{14}  & 8, 9, 0, 0, 0 & 2, 0, 0, 1, 0 & 9^2_{19} & 1 - 2 q + q^2 - q^4 +2q^5\\ \hline
G^9_{15} & 8, 9, 0, 1, 0 & 2, 0, 0, 1, 0 & 9^2_{13} & 1 - 2 q + q^2 + q^3 - q^4 \\ \hline
G^9_{16}  & 8, 9, 0, 0, 0 & 2, 0, 0, 0, -3 & 9_{35} & 1 - 2q + q^2 +3q^5 \\ \hline
\end{array}
$$
}
\caption{Figure \ref{f.compute} continued.
}
\lbl{f.compute3}
\end{figure}

\begin{figure}[!htpb]
{\small
$$
\begin{array}{|c|c|l|l|l|} \hline
G^{10}_0 & 6, 10, 5, 2, 1 & 5, 5, 5, 6, 11 & 10_{121} & 1 - 5 q + 5 q^2 + 10 q^3 - 16 q^4 - 7 q^5 \\ \hline
G^{10}_1 & 6, 10, 5, 0, 0 & 5, 5, 5, 6, 10 & 10_{123} & 1 - 5 q + 5 q^2 + 10 q^3 - 16 q^4 - 6 q^5 \\ \hline
G^{10}_2 & 6, 10, 4, 4, 0 & 5, 4, 0, -8, -20  & 10^4_{17} & 1 - 5 q + 6 q^2 + 10 q^3 - 21 q^4 - 11 q^5 \\ \hline
G^{10}_3 & 6, 10, 4, 4, 0 & 5, 4, 0, -8, -20  & 10^2_{155} & 1 - 5 q + 6 q^2 + 10 q^3 - 21 q^4 - 11 q^5 \\ \hline
G^{10}_4 & 6, 10, 4, 3, 0 & 5, 4, 1, -3, -6  & 10^2_{137} & 1 - 5 q + 6 q^2 + 9 q^3 - 21 q^4 - 6 q^5 \\ \hline
G^{10}_5 & 7, 10, 0, 10, 0 & 4, 0, -10, -20, -15  & 10^5_1 & 1 - 5 q + 6 q^2 + 9 q^3 - 21 q^4 - 6 q^5 \\ \hline

G^{10}_6 & 7, 10, 0, 8, 0 & 4, 0, -8, -13, -7  & 10^3_{25} & 1 - 4 q + 6 q^2 + 4 q^3 - 18 q^4 + 3 q^5 \\ \hline

G^{10}_7 & 7, 10, 0, 7, 0 & 4, 0, -7, -10, -5  & 10_{120} & 1 - 4 q + 6 q^2 + 3 q^3 - 17 q^4 + 7 q^5 \\ \hline

G^{10}_8 & 7, 10, 2, 2, 0 & 4, 2, 0, 1, 3  & 10^2_{33} & 1 - 4 q + 4 q^2 + 4 q^3 - 11 q^4 + 5 q^5 \\ \hline

G^{10}_9 & 7, 10, 3, 0, 0 & 4, 3, 3, 4, 3  & 10_{112} & 1 - 4 q + 3 q^2 + 5 q^3 - 6 q^4 + 4 q^5 \\ \hline

G^{10}_{10} & 7, 10, 2, 2, 0 & 4, 2, 0, 0, -1  &  10_{116} & 1 - 4 q + 4 q^2 + 4 q^3 - 10 q^4 + 5 q^5 \\ \hline
G^{10}_{11} & 7, 10, 1, 3, 0 & 4, 1, -2, 1, 7  & 10^2_{151} & 1 - 4 q + 5 q^2 + 2 q^3 - 14 q^4 + 11 q^5 \\ \hline

G^{10}_{12} & 7, 10, 1, 4, 0 & 4, 1, -3, -3, 0  &10_{119} & 1 - 4 q + 5 q^2 + 3 q^3 - 14 q^4 + 7 q^5 \\ \hline

G^{10}_{13} & 7, 10, 2, 2, 0 & 4, 2, 0, 0, -1  & 10_{114} & 1 - 4 q + 4 q^2 + 4 q^3 - 10 q^4 + 5 q^5 \\ \hline

G^{10}_{14} & 7, 10, 1, 3, 0 & 4, 1, -2, 0, 3  & 10^2_{156} & 1 - 4 q + 5 q^2 + 2 q^3 - 13 q^4 + 11 q^5 \\ \hline

G^{10}_{15} & 7, 10, 2, 1, 0 & 4, 2, 1, 4, 7  & 10^2_{147} & 1 - 4 q + 4 q^2 + 3 q^3 - 10 q^4 + 9 q^5 \\ \hline

G^{10}_{16} & 7, 10, 2, 1, 0 & 4, 2, 1, 3, 3  & 10_{122} & 1 - 4 q + 4 q^2 + 3 q^3 - 9 q^4 + 9 q^5 \\ \hline

G^{10}_{17} & 7, 10, 2, 2, 0 & 4, 2, 0, 0, -1  & 10^3_{74} & 1 - 4 q + 4 q^2 + 4 q^3 - 10 q^4 + 5 q^5 \\ \hline

G^{10}_{18} & 7, 10, 1, 4, 0 & 4, 1, -3, -2, 4  & 10^2_{28} & 1 - 4 q + 5 q^2 + 3 q^3 - 15 q^4 + 7 q^5 \\ \hline

G^{10}_{19} & 7, 10, 2, 1, 0 & 4, 2, 1, 5, 11  &10^4_{12} & 1 - 4 q + 4 q^2 + 3 q^3 - 11 q^4 + 9 q^5 \\ \hline

G^{10}_{20} & 8, 10, 0, 1, 0 & 3, 0, -1, 3, 4  & 10^2_{106} & 1 - 3 q + 3 q^2 - 6 q^4 + 8 q^5 \\ \hline

G^{10}_{21} & 8, 10, 0, 1, 0 & 3, 0, -1, 2, 1  & 10^2_{20} & 1 - 3 q + 3 q^2 - 5 q^4 + 8 q^5 \\ \hline

G^{10}_{22} & 8, 10, 0, 1, 0 & 3, 0, -1, 1, -1  & 10^2_{141}& 1 - 3 q + 3 q^2 - 4 q^4 + 7 q^5 \\ \hline

G^{10}_{23} & 8, 10, 0, 1, 0 & 3, 0, -1, 2, 2  & 10_{93} & 1 - 3 q + 3 q^2 - 5 q^4 + 7 q^5 \\ \hline

G^{10}_{24} & 8, 10, 1, 1, 0 & 3, 1, 0, 0, -1  & 10_{85} & 1 - 3 q + 2 q^2 + 2 q^3 - 3 q^4 + 2 q^5 \\ \hline

G^{10}_{25} & 8, 10, 0, 2, 0 & 3, 0, -2, -1, -1  &10_{100} & 1 - 3 q + 3 q^2 + q^3 - 5 q^4 + 4 q^5 \\ \hline

G^{10}_{26} & 8, 10, 1, 0, 0 & 3, 1, 1, 3, 3  & 10^3_{33} & 1 - 3 q + 2 q^2 + q^3 - 3 q^4 + 5 q^5 \\ \hline

G^{10}_{27} & 8, 10, 2, 0, 0 & 3, 2, 2, 2, 2  & 10^3_{40} & 1 - 3 q + q^2 + 3 q^3 - q^4 + q^5 \\ \hline

G^{10}_{28} & 8, 10, 0, 3, 0 & 3, 0, -3, -4, -5  & 10^2_{59} & 1 - 3 q + 3 q^2 + 2 q^3 - 5 q^4 + 2 q^5 \\ \hline

G^{10}_{29} &8, 10, 0, 1, 0  & 3, 0, -1, 3, 5  & 10^3_{37} & 1 - 3 q + 3 q^2 - 6 q^4 + 7 q^5 \\ \hline

G^{10}_{30} &8, 10, 0, 0, 0  & 3, 0, 0, 4, 2  & 10^2_{37} & 1 - 3 q + 3 q^2 - q^3 - 4 q^4 + 10 q^5 \\ \hline

G^{10}_{31} &8, 10, 1, 0, 0  & 3, 1, 1, 2, 0  & 10_{108} & 1 - 3 q + 2 q^2 + q^3 - 2 q^4 + 5 q^5 \\ \hline

G^{10}_{32} &8, 10, 0, 2, 0  & 3, 0, -2, -2, -5  & 10^2_{40} & 1 - 3 q + 3 q^2 + q^3 - 4 q^4 + 5 q^5 \\ \hline

G^{10}_{33} &8, 10, 0, 2, 0  & 3, 0, -2, -2, -6  & 10^4_3 & 1 - 3 q + 3 q^2 + q^3 - 4 q^4 + 6 q^5 \\ \hline

G^{10}_{34} &8, 10, 0, 3, 0  & 3, 0, -3, -4, -6  & 10^4_{22} & 1 - 3 q + 3 q^2 + 2 q^3 - 5 q^4 + 3 q^5 \\ \hline

G^{10}_{35} &8, 10, 0, 2, 0  & 3, 0, -2, -2, -6  & 10^4_3 & 1 - 3 q + 3 q^2 + q^3 - 4 q^4 + 6 q^5 \\ \hline

G^{10}_{36} &9, 10, 0, 1, 0  & 2, 0, -1, -1, -1  &  10^3_{69}& 1 - 2 q + q^2 + q^3 - q^4 \\ \hline

G^{10}_{37} &9, 10, 0, 0, 0  & 2, 0, 0, 1, 1  & 10_{46} & 1 - 2 q + q^2 - q^4 + q^5 \\ \hline

G^{10}_{38} &9, 10, 0, 0, 0  & 2, 0, 0, 0, -2  & 10^3_{65} & 1 - 2 q + q^2 + 2 q^5 \\ \hline

G^{10}_{39} &9, 10, 0, 0, 0  & 2, 0, 0, 0, -1  & 10_{61} & 1 - 2 q + q^2 + q^5 \\ \hline

G^{10}_{40} &10, 10, 0, 0, 0  & 1, 0, 0, 0, 0  & 10^2_1 & 1 - q \\ \hline

\end{array}
$$
}
\end{figure}

\begin{figure}[!htpb]
{\small
$$
\begin{array}{|c|l|l|l|} \hline
G & C & L & \Phi_L(q) + O(q^6) \\ \hline
Gv^6_1 & 3, 0, -6, -10, -7 & 8^4_1 &  1 - 3 q + 3 q^2 + 5 q^3 - 8 q^4 - 5 q^5  \\ \hline  Gv^6_2 &  2, 0, -1, 1, 2 & 7^2_4 &  1-2 q+q^2+q^3-3 q^4+q^5 \\ \hline
Gv^6_3 & 4, 2, -3, -9, -13 & 9^3_{12} & 1 - 4 q + 4 q^2 + 7 q^3 - 13 q^4 - 7 q^5  \\ \hline
Gv^6_4 & 1, 0, 0, 0, -1 & 6^2_1 &  1 - q + q^5 \\ \hline
Gv^6_5 & 3, 2, 2, 4, 6 & 8^3_5 & 1 - 3 q + q^2 + 3 q^3 - 3 q^4 + 3 q^5  \\ \hline
Gv^6_6 & 4, 3, 2, 3, 6 & 9^2_{42} &  1 - 4 q + 3 q^2 + 6 q^3 - 9 q^4 \\ \hline
Gv^6_7 & 5, 5, 5, 6, 11 & 10_{121} & 1 - 5 q + 5 q^2 + 10 q^3 - 16 q^4 - 7 q^5  \\ \hline
Gv^6_8 & 5, 5, 5, 6, 10 & 10_{123} &  1 - 5 q + 5 q^2 + 10 q^3 - 16 q^4 - 6 q^5\\ \hline
Gv^6_9 & 5, 4, 0, -8, -20 & 10^4_{17} & 1 - 5 q + 6 q^2 + 10 q^3 - 21 q^4 - 11 q^5  \\ \hline
Gv^6_{10} & 3, 1, -1, 0, 2 & 8_{16} &  1 - 3 q + 2 q^2 + 3 q^3 - 6 q^4 + q^5\\ \hline
Gv^6_{11} &  4, 2, -2, -5, -5 & 9_{34} & 1 - 4 q + 4 q^2 + 6 q^3 - 13 q^4 - 3 q^5  \\ \hline
Gv^6_{12} & 5, 4, 0, -8, -20 & 10^2_{155} &  1 - 5 q + 6 q^2 + 10 q^3 - 21 q^4 - 11 q^5\\ \hline
Gv^6_{13} & 5, 4, 0, -8, -20 & 9_{40} & 1 - 4 q + 4 q^2 + 5 q^3 - 13 q^4 + q^5  \\ \hline
Gv^6_{14} & 5, 4, 0, -8, -20  & 10^2_{137} &  1 - 5 q + 6 q^2 + 9 q^3 - 21 q^4 - 6 q^5\\ \hline
Gv^6_{15} & 6, 7, 8, 8, 9 & 11_{314} & 1 - 6 q + 8 q^2 + 14 q^3 - 29 q^4 - 17 q^5  \\ \hline Gv^6_{16} & 6, 6, 3, -7, -28 & L11a520 &  1 - 6 q + 9 q^2 + 13 q^3 - 35 q^4 - 17 q^5\\ \hline
Gv^6_{17} & 7, 10, 16, 25, 46 & L12a1183 & 1 - 7 q + 11 q^2 + 19 q^3 - 43 q^4 - 33 q^5  \\ \hline Gv^6_{18} & 7, 8, 5, -13, -65 & L12a2008 &  1 - 7 q + 13 q^2 + 16 q^3 - 57 q^4 - 28 q^5\\ \hline
Gv^6_{19} & 3, 0, -5, -7, -4 & 8^2_{14} & 1 - 3 q + 3 q^2 + 4 q^3 - 8 q^4 - 2 q^5 \\ \hline
\end{array}
$$
}
\caption{The irreducible graphs $G$ with 6 vertices, the vector $C=
(C_1,\dots,C_5)$, the
alternating link $L$ and the 6 stable coefficients of the Jones polynomial
of $L$. }
\lbl{f.Gvcompute}
\end{figure}


\section{Tables of irreducible planar graphs}
\lbl{sec.app}

\begin{figure}[!htpb]
$$
\psdraw{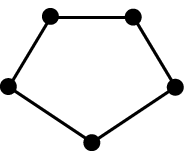}{0.7in} \qquad \psdraw{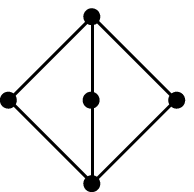}{0.7in} \qquad \psdraw{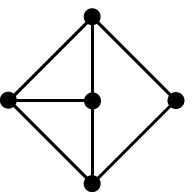}{0.7in}
\qquad \psdraw{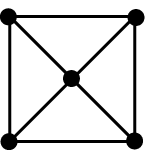}{0.7in} \qquad \psdraw{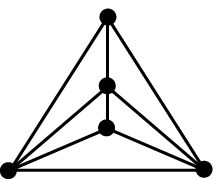}{0.7in}
$$
\caption{The irreducible planar graphs $Gv^5_i$ for $i=1,\dots,5$
(from the left to the right) with 5 vertices.}
\lbl{f.Gv5}
\end{figure}

\begin{figure}[!htpb]
$$
\psdraw{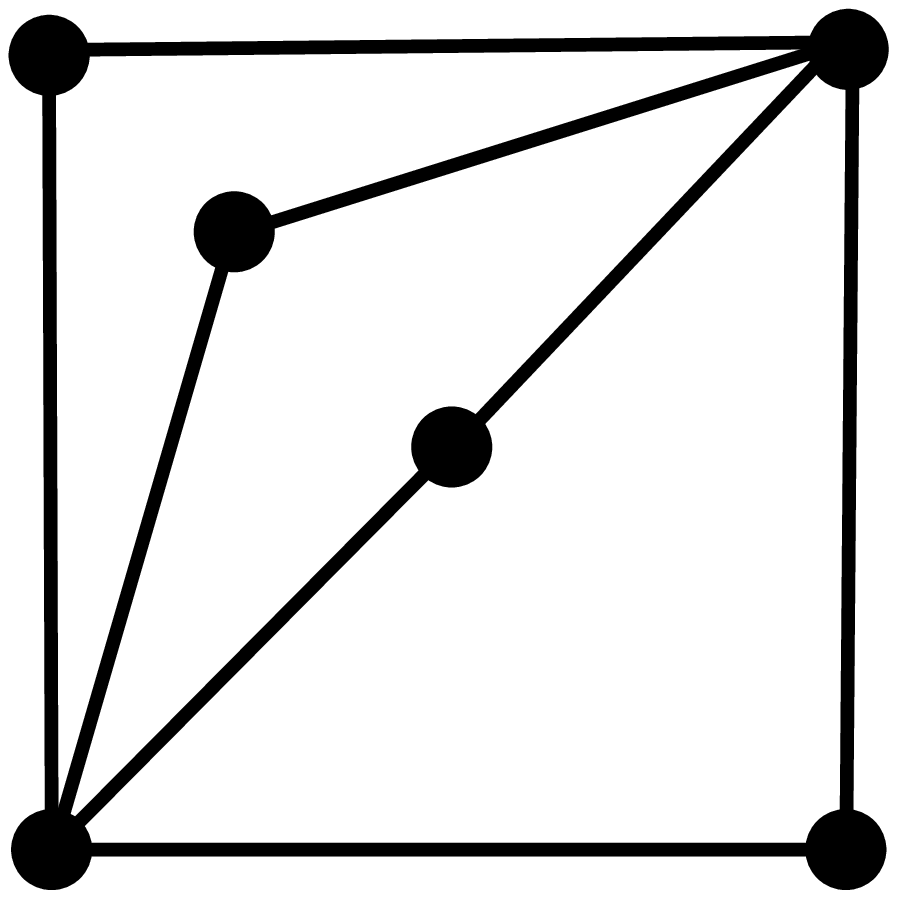}{0.7in} \qquad \psdraw{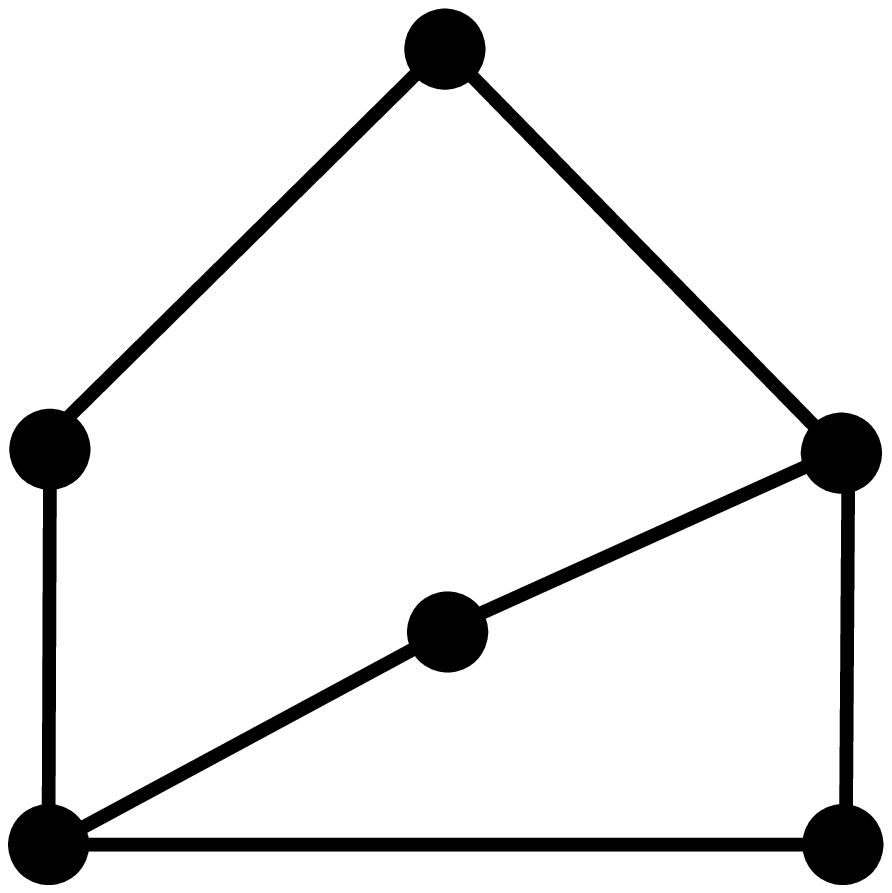}{0.6in} \qquad \psdraw{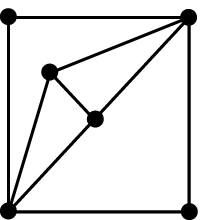}{0.6in}
\qquad \psdraw{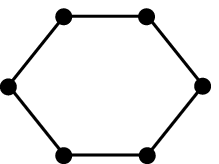}{0.7in} \qquad \psdraw{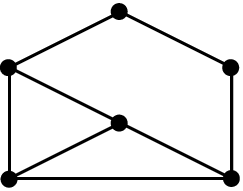}{0.7in}
$$
$$
\psdraw{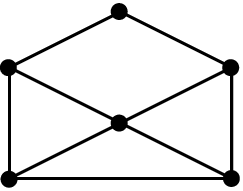}{0.7in} \qquad \psdraw{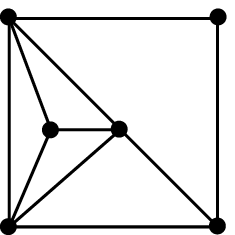}{0.6in}\qquad \psdraw{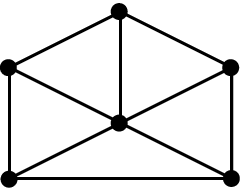}{0.7in} \qquad \psdraw{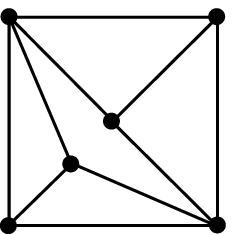}{0.6in}\qquad \psdraw{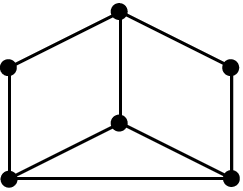}{0.7in}
$$
$$
 \psdraw{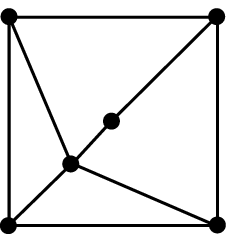}{0.6in}\qquad \psdraw{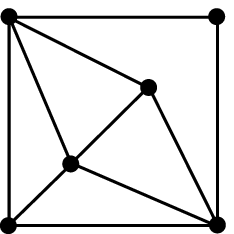}{0.6in}\qquad \psdraw{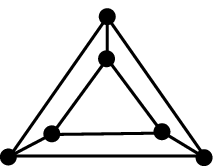}{0.7in}\qquad \psdraw{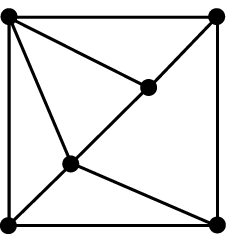}{0.6in}\qquad \psdraw{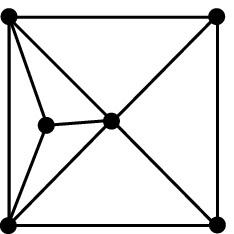}{0.6in}
$$
$$
\psdraw{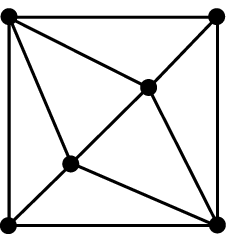}{0.6in} \qquad \psdraw{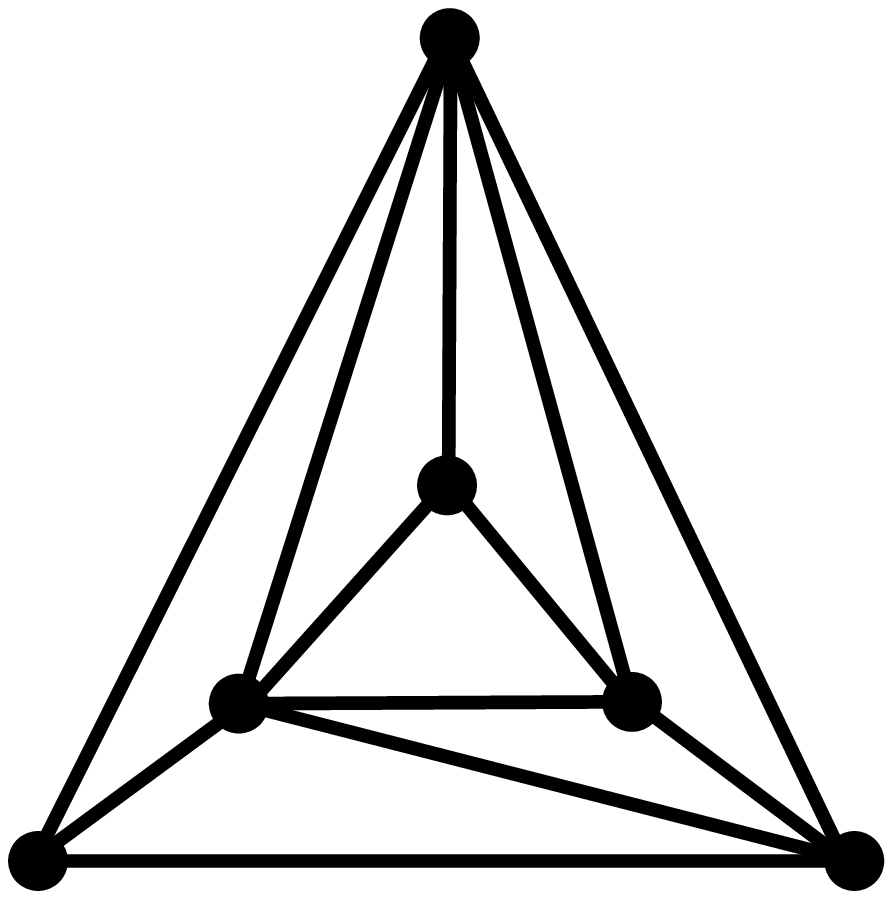}{0.7in}\qquad \psdraw{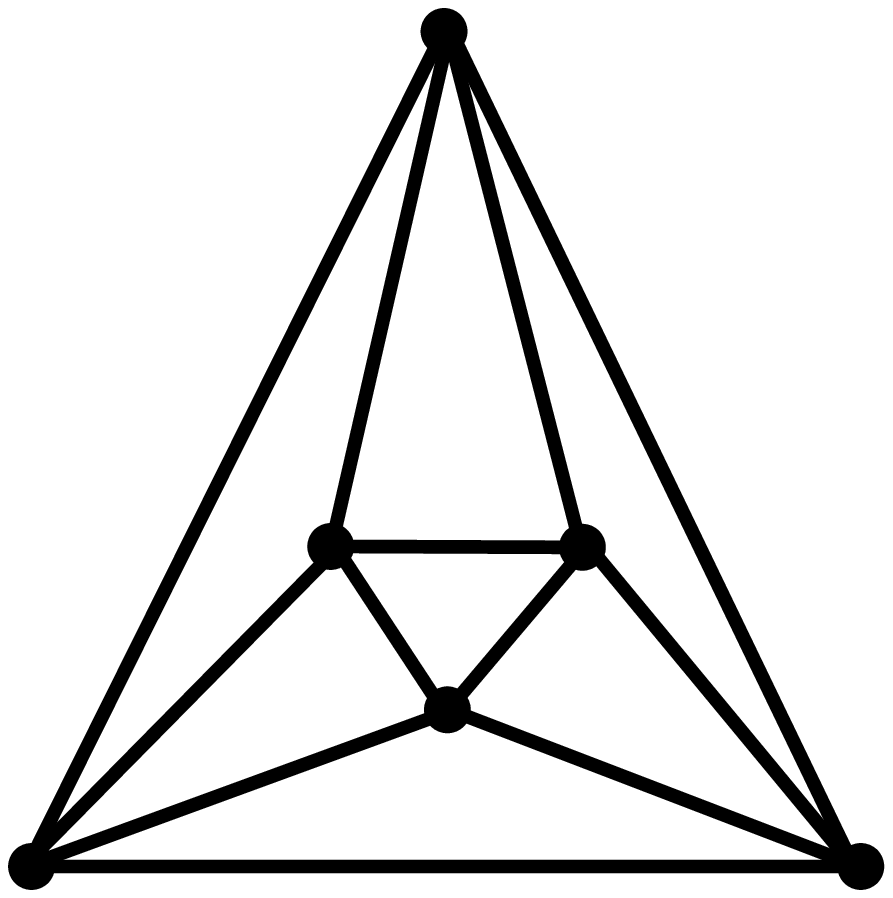}{0.7in}\qquad \psdraw{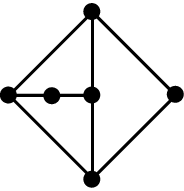}{0.7in}
$$
\caption{The irreducible planar graphs $Gv^6_i$ for $i=1,\dots,19$
(from the left to the right) with 6 vertices.}
\lbl{f.Gv6}
\end{figure}

\begin{figure}[!htpb]
$$
\psdraw{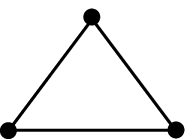}{0.7in} \qquad \psdraw{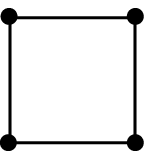}{0.55in} \qquad
\psdraw{50}{0.7in}
$$
\caption{The irreducible planar graphs $G_0^3,G_0^4$
and $G_0^5$ with 3, 4 and 5 edges.}
\lbl{f.graphs345}
\end{figure}

\begin{figure}[!htpb]
$$
\psdraw{60}{0.7in} \qquad  \psdraw{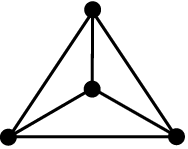}{0.7in} \qquad \psdraw{62}{0.7in}
$$
$$
\psdraw{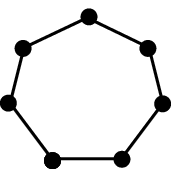}{0.7in} \qquad  \psdraw{71}{0.7in} \qquad \psdraw{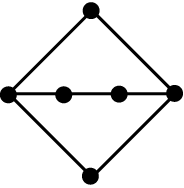}{0.7in}
$$
\caption{The irreducible planar graphs with 6 and 7 edges: $G_0^6,G_1^6,G_2^6$ on the top and $G_0^7,G_1^7,G_2^7$ on the bottom.}
\lbl{f.graphs67}
\end{figure}

\begin{figure}[htpb]
$$
\psdraw{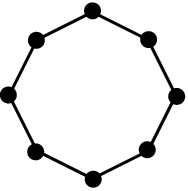}{0.7in} \qquad \psdraw{81}{0.7in} \qquad \psdraw{82}{0.7in} \qquad
\psdraw{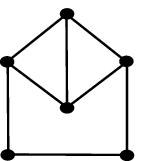}{0.7in}
$$
$$
 \psdraw{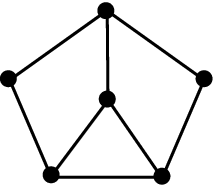}{0.7in}  \qquad \psdraw{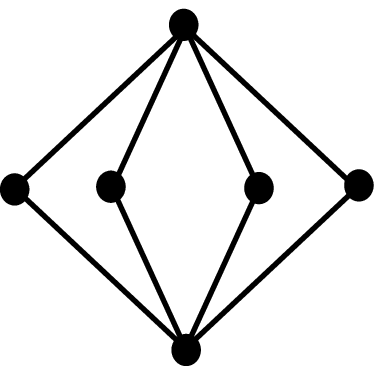}{0.8in} \qquad \psdraw{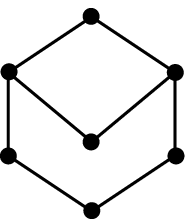}{0.7in} \qquad
\psdraw{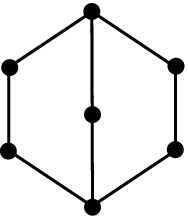}{0.7in}
$$
\caption{The irreducible planar graphs with 8 edges: $G_0^8,\dots,G_3^8$ on the top (from left to right) and $G_4^8,\dots,G_7^8$ on the bottom.}
\lbl{f.graphs8}
\end{figure}

\begin{figure}[!htpb]
$$
 \psdraw{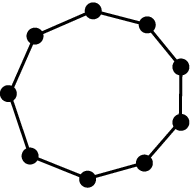}{0.7in} \quad
\psdraw{91}{0.7in} \quad
\psdraw{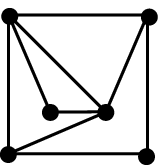}{0.7in} \quad \psdraw{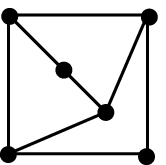}{0.7in} \quad
\psdraw{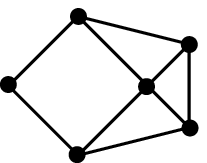}{0.7in} \quad \psdraw{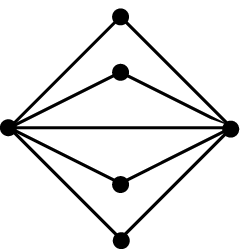}{0.8in}
$$
$$
\psdraw{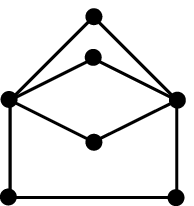}{0.7in}  \quad
  \psdraw{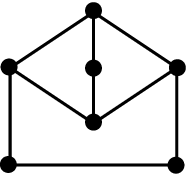}{0.7in}  \quad
\psdraw{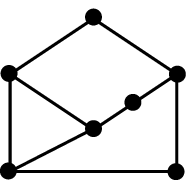}{0.7in} \quad \psdraw{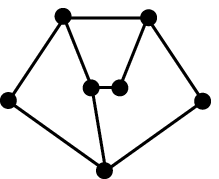}{0.8in} \quad
 \psdraw{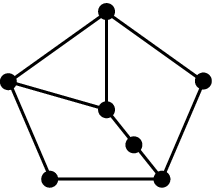}{0.7in} \quad \psdraw{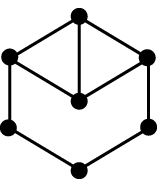}{0.7in}
$$
$$
 \psdraw{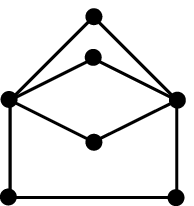}{0.7in} \quad
 \psdraw{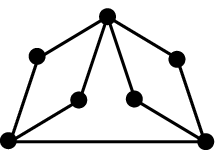}{0.7in} \quad \psdraw{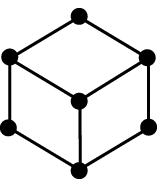}{0.7in} \quad \psdraw{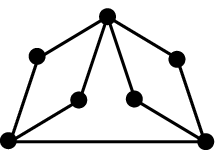}{0.7in} \quad
\psdraw{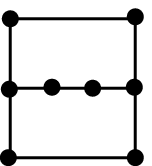}{0.7in}
$$
\caption{The irreducible planar graphs with 9 edges: $G_0^9,\dots,G_5^9$ on the top, $G_6^9,\dots,G_{11}^9$ on the middle and $G_{12}^{9},\dots,G_{16}^9$ on the bottom.}
\lbl{f.graphs9}
\end{figure}

\begin{figure}[!htpb]
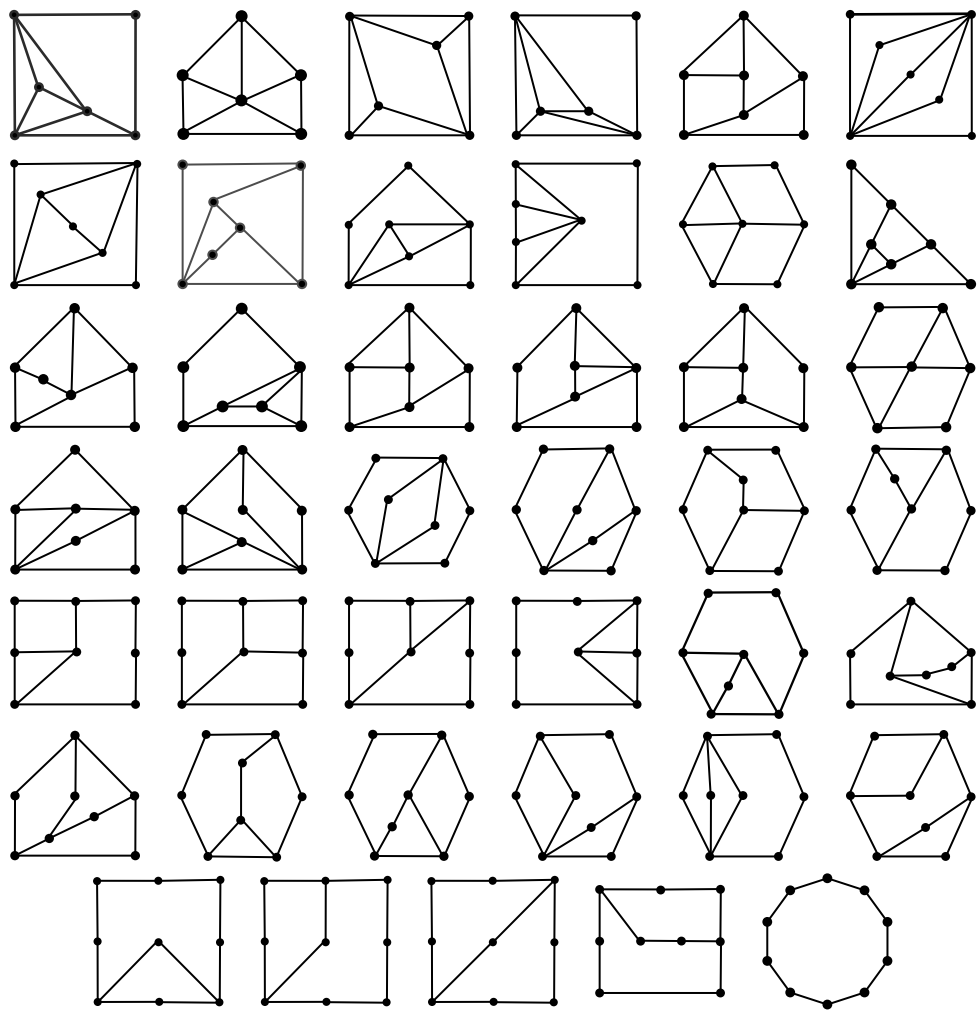

$$
 \psdraw{{GG10.0}}{0.7in} \quad
 \psdraw{{GG10.1}}{0.7in}\quad
 \psdraw{{GG10.2}}{0.7in} \quad
\psdraw{{GG10.3}}{0.7in} \quad
 \psdraw{{GG10.4}}{0.7in} \quad
\psdraw{{GG10.5}}{0.7in}
$$
$$
 \psdraw{{GG10.6}}{0.7in} \quad
 \psdraw{{GG10.7}}{0.7in}\quad
 \psdraw{{GG10.8}}{0.7in} \quad
\psdraw{{GG10.9}}{0.7in} \quad
 \psdraw{{GG10.10}}{0.7in} \quad
\psdraw{{GG10.11}}{0.7in}
$$
$$
 \psdraw{{GG10.12}}{0.7in} \quad
 \psdraw{{GG10.13}}{0.7in}\quad
 \psdraw{{GG10.14}}{0.7in} \quad
\psdraw{{GG10.15}}{0.7in} \quad
 \psdraw{{GG10.16}}{0.7in} \quad
\psdraw{{GG10.17}}{0.7in}
$$
$$
 \psdraw{{GG10.18}}{0.7in} \quad
 \psdraw{{GG10.19}}{0.7in}\quad
 \psdraw{{GG10.20}}{0.7in} \quad
\psdraw{{GG10.21}}{0.7in} \quad
 \psdraw{{GG10.22}}{0.7in} \quad
\psdraw{{GG10.23}}{0.7in}
$$
$$
 \psdraw{{GG10.24}}{0.7in} \quad
 \psdraw{{GG10.25}}{0.7in}\quad
 \psdraw{{GG10.26}}{0.7in} \quad
\psdraw{{GG10.27}}{0.7in} \quad
 \psdraw{{GG10.28}}{0.7in} \quad
\psdraw{{GG10.29}}{0.7in}
$$
$$
 \psdraw{{GG10.30}}{0.7in} \quad
 \psdraw{{GG10.31}}{0.7in}\quad
 \psdraw{{GG10.32}}{0.7in} \quad
\psdraw{{GG10.33}}{0.7in} \quad
 \psdraw{{GG10.34}}{0.7in} \quad
\psdraw{{GG10.35}}{0.7in}
$$
$$
 \psdraw{{GG10.36}}{0.7in} \quad
 \psdraw{{GG10.37}}{0.7in}\quad
 \psdraw{{GG10.38}}{0.7in} \quad
\psdraw{{GG10.39}}{0.7in} \quad
 \psdraw{{GG10.40}}{0.7in}
$$
\caption{The irreducible planar graphs with 10 edges: $G_0^{10},\dots,G_5^{10}$ on the top, $G_6^{10},\dots,G_{35}^{10}$ on the middle and $G_{36}^{10},\dots,G_{40}^{40}$ on the bottom.}
\lbl{f.graphs10}
\end{figure}


\bibliographystyle{hamsalpha}
\bibliography{biblio}
\end{document}